\documentclass[a4paper]{article}
\usepackage[a4paper,margin=3cm]{geometry}
\usepackage[utf8]{inputenc}
\usepackage{amsthm,amssymb,amsbsy,amsmath,amsfonts,amssymb,amscd}
\usepackage{amsmath}
\usepackage{amssymb}
\usepackage{xcolor}
\usepackage{graphicx}
\usepackage{subcaption}

\usepackage{subfiles}
\usepackage{csquotes} 
\usepackage[todonotes={textsize=scriptsize}]{changes}
\usepackage[colorlinks=true,linkcolor=blue, citecolor=red]{hyperref}
\usepackage{color}

\usepackage{microtype}

\newcommand{\vertiii}[1]{{\left\vert\kern-0.25ex\left\vert\kern-0.25ex\left\vert #1 
    \right\vert\kern-0.25ex\right\vert\kern-0.25ex\right\vert}}
\newcommand{\vincent}[1]{\textcolor{red}{#1}}

\newcommand{\E}{\mathbb E}
\newcommand{\R}{\mathbb R}
\newcommand{\Pp}{\mathbb P}
\theoremstyle{plain}
\newtheorem{theorem}{Theorem}[section]
\newtheorem{prop}[theorem]{Proposition}
\newtheorem{lemme}[theorem]{Lemma}

\newtheorem{rque}[theorem]{Remark}
\newcommand{\be} {\begin{equation}}
\newcommand{\ee} {\end{equation}}
\newcommand{\bea} {\begin{eqnarray}}
\newcommand{\eea} {\end{eqnarray}}
\newcommand{\Bea} {\begin{eqnarray*}}
\newcommand{\Eea} {\end{eqnarray*}}

\title{ Growing random planar network with \\oriented branching and fusion\\
}
\author{Vincent Bansaye\thanks{CMAP, CNRS, INRIA Merge, École polytechnique, Institut Polytechnique de Paris, Palaiseau, France, \\vincent.bansaye@polytechnique.edu, gael.raoul@polytechnique.edu, milica.tomasevic@polytechnique.edu }, Gael Raoul\footnotemark[1]\,{ }\thanks{International Research Laboratory France-Vietnam in Mathematics and its Applications, CNRS - VAST - VIASM}, Milica Tomasevic\footnotemark[1]}
\date{}

\begin{document}
\maketitle
\begin{abstract}
We consider a growing planar network 
where a tip grows at constant speed, branches at constant rate and inactivates when it meets a branch already created. We only consider here orthogonal branching occurring always in the same direction.  This yields a  spatial branching property to the growing network. The connected components of the network then form a  branching process of rectangles   with double immigration. Using a spine approach for a typical rectangle and coupling arguments, the study  is boiled down to a one dimensional stick breaking model with aging. We can then  prove long time convergence of empirical measure of  the family of rectangles after polynomial rescaling. The limiting  distribution and speed of convergence  can be explicitly described. The proofs also rely on the description of common ancestor of rectangles in the branching structure with double immigration.  
\end{abstract}



\tableofcontents
\section{Introduction}

In this work, we consider a growing planar network embedded in 
$\mathbb{R}^2$. The network is composed of straight branches, whose growth occurs at their active tips. Each active tip propagates along a straight line at constant speed. Branching events occur according to a fixed rate, with each active tip undergoing binary branching into branches with prescribed directions. A tip ceases growth when it intersects any previously created part of the network, at which point it becomes inactive. The resulting structure is a planar embedded network with non-overlapping growth. One may think of the construction of road or railway networks featuring bifurcations and non-overlapping paths. Our primary motivation, however, comes from the growth of living systems, such as organ morphogenesis and mycelial networks, which will be discussed in more detail below.\\
The inactivation of tips upon meeting the existing network induces interactions between branches and breaks the classical branching property. As a result, the process is related to branching–annihilating random walks \cite{Birkneral,Bramson, CardyT}. At the same time, the non-overlapping constraint between branches induces independence between distinct connected components and preserves an underlying branching structure. This structure becomes particularly  tractable and nice when branching events are orthogonal and share a common orientation. In this work, we focus on this simplified setting, which can be viewed as a toy model serving as a building block for more general growth processes. In this regime, one can analyze in detail, and in a largely explicit manner, the long-time distributions and the genealogical structure of the network, including the sizes of the rectangles composing it and the relationships between their common ancestors. Possible extensions of the model are briefly discussed in the Discussion section.\\
Moreover, by simple rescaling of time and space, the two parameters—the elongation rate and the branching rate—can be set to $1$ without loss of generality. The problem is also  invariant by rotation, and the network exhibits additional symmetries when growth starts from active tips in four orthogonal directions at the same time. Exploiting these symmetries, we can restrict the study to the network in a single quadrant and focus on the following model:
\begin{itemize}
    \item The process is initiated by a single active tip in $(0,0)$, oriented along the $x$-axis, i.e. along to vector $(1,0)$. Each active tip elongates its branch at speed $1$. This growth occurs in a constant direction, so that the branches form straight segments.
    \item The active tips become inactive if they encounter an existing branch, or the $y$-axis $\{x=0\}$.
    \item Each active tip produces an offspring active tip with rate $1$. Then, the mother tip continues its elongation unperturbed (\textit{straight offspring}), while the daughter active tip spurs from the mother tip at an oriented angle $\pi/2$ compared to the mother tip (\textit{orthogonal offspring}).
\end{itemize}

We will show that the growth of this network can be mapped to a branching process in which the particles correspond to aging, potentially unbounded rectangles. This leads to a structured branching process with types in $[0,\infty]^3$, corresponding to two sizes and an age along one side. Each rectangle can either split into two particles (fragmentation) or become frozen (inactivation). This framework facilitates the analysis of the network at large times, in particular the description of statistical properties and the genealogical structure. Rectangles of infinite size correspond to active tips at the boundary of the network and induce a (double) immigration mechanism in the branching process, which will be described in detail later. In addition, the growth of active tips naturally defines the age of the corresponding rectangle.

The analysis of the model will thus be carried out by a fragmentation process with  aging and immigration. Surface preservation of rectangles at fragmentation leads us to consider the $h$ transform associated to the surface and a spine approach, corresponding to a size biased typical rectangle, in the vein of the works of  Jean Bertoin and  B\'en\'edicte Hass and their tagged fragment, see in particular \cite{Bertoin, DJ} and references therein. This reduces the study of the network  to a Markov process in $\R_+^3$ describing a typical aging rectangle. By  coupling arguments, we can split the length-width of rectangles. We thus actually boil down  the analysis of the first moment semigroup of the branching process to a Markov process in $\R_+^2$ describing a  stick breaking scheme. This will give  the main ingredients of the proofs, allowing for quantitative estimates and convergence and also an analytic expansion of the stationary distribution of the  sizes.  
Related works also include studies on spatially structured fragmentation \cite{growth}. The main differences in our model are the aging mechanism, which induces freezing, and the immigration structure arising from infinite components. We also note the considerable activity in growth-fragmentation processes over the past two decades (see, e.g.,~\cite{Bertoin, BCK, Dadoun, SilvaPardo}). In contrast to classical growth-fragmentation, growth in our process originates from the spatial structure and occurs at the boundary of the network, while growth-fragmentation models typically combine growth for each particle between branching events.

More formally,  we study  the empirical measure of rectangles 
$$Z_t=\sum_{u\in V(t)\cup W(t) } \delta_{(L_u, \ell_u, A_u(t))},$$
where $V(t)$ (resp. $W(t)$) is the set  of active (resp. inactivated) rectangles created by the network at   time $t$, $(L_u, \ell_u)$ their length and width,  and $A_u(t)$ is the age corresponding to the progression of the network along the ``first side" $L_u$ of the rectangle. Let us  observe that $A_u(t)=L_u$ yields the set $u$ of inactivated rectangles at time $t$. We refer to Section~\ref{constructdef} for precise definitions and constructions. There we show that the spatial growing network can be described using the dynamics of this family of rectangles, which enjoys a branching property.
We then analyze the first moment semigroup of $Z$, defined by
\begin{equation}
\label{sgpremiermoment}
M_tf=\E(Z_t(f))=\E\left(\sum_{u\in V(t)\cup W(t)} f(L_u, \ell_u, A_u(t)) \right) \quad  t\geq 0.
\end{equation} 
As mentioned above, the main tool will be a spinal rectangle defined through  $h$ transform, with $h(L,\ell,a)=L\ell$ being the surface of a given rectangle.
For rectangles with finite width and length, this  allows to reduce the problem 
to a  stick fragmentation with aging, see Section \ref{spinal}. In Section \ref{Firstmomentsg}, we complement these results on the first moment of finite rectangles by adding the double immigration, generated by  infinite rectangles. In particular, we show that the set of inactivated (resp. active) rectangles is of order $t^2$ (resp. $t$) at time $t$, when $t$ goes to infinity.  We then exploit the genealogical structure of branching process with immigration to prove a law of large number. Thus, in Section \ref{laviedansL2}, we prove that
for any test function $f$ which grows sub-exponentially, the following convergence holds in $L^2$ (see Theorem \ref{Thprinc}):
$$\frac{\langle Z_t, f \rangle}{t^2} \stackrel{t\rightarrow\infty}{\longrightarrow}\frac{\Pi(f)}{2 \Pi(h)},$$
where $\Pi$ admits a density for the two first coordinates (length and width). The latter is given by
$$\sum_{k\geq 1} \frac{k+1}{((k-1)!)^2} \left( ke^{-(k+1)L-kl} +\frac{1+k}{k}e^{-(1+k)L-(1+k)l}\right).$$
Indeed, the measure $\Pi$ is concentrated on inactivated rectangles, whose age coincides with the length. It can be characterized through a fixed point problem, which allows to prove that it is absolutely continuous.
Moreover,  we  estimate the    speed  of convergence in $L^2$ and prove that it decreases (at least) as $1/t$.  \\
In Section~\ref{sec:discussion} we conclude with a discussion on mathematical and modelling perspectives of our work. \\
Finally, in the Appendix~\ref{sec:APP1} we describe the  corresponding partial differential equation (PDE), i.e. the PDE associated by duality to the generator of the first moment semigroup (see Eq.~\eqref{eq:nmp}). It is related to growth-fragmentation equations (see e.g. \cite{fournier2006well}) and to renewal equations (see Chapter 3 of \cite{perthame2007transport}). Population models combining size parameters and an age have been introduced in \cite{olivier2016does,doumic2020purely,gabriel2018steady}, to describe precisely the duplication dynamics of bacteria. We are not aware, however, of models similar to the equation satisfied by the mean measure of rectangles with both length and width finite, where the age structure has a direct impact on the fragmentation affecting the size parameters.
To conclude in Appendix~\ref{sec:APP2} we discuss the numerical simulations that we used to produce all the figures in this work. 



\paragraph{Literature and motivations from biology.} Formation of biological networks and their spatial structure is  of wide interest in many different contexts. Non-overlapping planar networks appear in many fields \cite{Bartelemy}, among which crystal dendrite formation \cite{Zawko}, urban development \cite{Courta} or mycelial networks \cite{Tomasevic, Dikec}. In particular, such networks appear in leaf venation. Leaf veins are very easily visible, and a great diversity of structures appear. They typically contain loops, a property that is explained by a way to enhance the resilience of the network \cite{Katifori}, and contain fractal-like structures. This structure can be used for species classification, and it has a direct impact on the plant phenotypes, such as its ability to resist water stress \cite{Scoffoni}.  They have therefore been extensively studied, with several points of view. Following Turing pioneering ideas \cite{Turing}, models based on reaction-diffusion equations have been proposed to describe the dynamics of chemical signals guiding the formation of vessels \cite{Meinhardt,Hogan}. Most models, however, are based on local or global optimization approach \cite{Matos, Ronellenfitsch}: The network would be structured to optimize objectives, such as fluid transport efficiency, resistance to drought, resilience to herbivore, etc.

Network formation models also play an essential role in organ morphogenesis. Here also, models based on extensions of Turing reaction-diffusion models have been proposed \cite{Sun,Menshykau}, relating the development of the branching structures of e.g. lungs and kidney thanks to reaction-diffusion equations. These models rely on a good understanding of local growth and branching mechanisms. On the contrary, the regulation mechanisms at a larger scale, that could explain the branching patterns observed, remain largely unknown \cite{Yu}. Over the past decade, simple macroscopic branching models have been introduced, see \cite{Hannezo}, in particular to describe the structure on organs having a non-overlapping $2d$ structure. Authors propose to rely on active tips that explore the plane, laying the vessels in their track, where the tips branch at a constant rate into two active tips, and turn inactive if they encounter a vessel \cite{HannezoOpinion}. These models have been introduced and developed by a team lead by Edouardo Hannezo, and lead to the formation of surprisingly relevant branching structures, with very few parameters. They could provide a unifying point of view on branching structures morphogenesis. Examples  of organs where this model shows promising results include \emph{mouse mammary glands},  \emph{pancreatic organoids}, \emph{kidney explants} \cite{HannezoOpinion}, and it may also be an interesting approach to study the dendritic architectures generated by some neuronal cells \cite{Fujishima,Zubler}. 

Another example are mycelial networks that play an important role  in the functioning of their ecosystem by decomposing the surrounding
organic matter. The mycelium consists of filamentous structures called hyphae that can grow and branch to create potentially huge network that can cover several kilometers. These dynamics are complicated to analyse as there is a lot of dependence in the network through nutrients sharing locally or through global signalization when the mycelium responds to stress (e.g. light, obstacle, predator). One important way to enhance the internal communication is the fusion of hyphae that may happen when a tip of one hyphe encounters the mycelial network. The mathematical difficulty then  comes from the assumption that active tips become inactive when they encounter a vessel: this implies an interaction between the tips and the other vessel it comes across (notice that vessels are the trajectories of the tips). Recently, models motivated by experimental setting \cite{Dikec} have arisen in the literature.  When  active tips are not affected by vessels (that is when vessels can cross), the dynamics of each active tip becomes independent, and the process becomes a \emph{branching process} in the sense of probability. It is then possible to derive precise properties of the structure \cite{Tomasevic}. If the interactions are less stiff, namely if the active tips can stop at a rate that is an increasing function of the number of vessels locally present, it is possible to derive macroscopic models: this has been done heuristically in \cite{Hannezo}, and a related rigorous result can be found in \cite{Catellier,Kuwata}, while a mathematical analysis of the macroscopic partial differential equations can be found in \cite{Kreten,Kreten2}. \\

{\bf Notation : constants and norms.} In this paper, we consider a model without any parameters. Our constants $C$ will thus be universal constant, which do not depend on time, nor on the initial sizes of the  rectangle,  nor on
the test functions. These constants may however change from one statement to another. \\
We first use the $L^{\infty}$ norm for finite rectangles.
It is denoted by $\parallel \cdot \parallel_{\infty}$.  We then need more complex norms to extend the space of test functions and take into account infinite rectangles. The good functional space for our work seems functions which are dominated by exponential growths. We thus consider the norm $\parallel \cdot \parallel_{\ell,\infty}$ when one size of the rectangle is less than $\ell$ and $\parallel \cdot \parallel$ without restrictions. They are involved in the description of the first moment semigroup and   defined by $$\parallel f \parallel_{\ell,\infty}=\sup_{\substack{(x,y,a,a')\in \mathbb R_+^4 \\  a\leq x,\,  \,  a'\leq y\leq \ell}} (\vert f(x,y,a)\vert +\vert f(y,x,a')\vert)e^{-3x/4}, \quad \parallel f\parallel =\sup_{\substack{ (x,a,y)  \in \R_+^3, \\ a\leq x}} \vert f(x,y,a) \vert e^{-3(x+y)/4}.$$
At the end, we obtain trajectorial result and  focus on   the empirical measure restricted to finite rectangles since infinite rectangles are easily described by Poisson immigration and less numerous. The norm we can then control is the following 
$$\vertiii f:= \sup_{\substack{(L,\ell,a,a')\in \R_+^4\\ a\leq L, a'\leq \ell} }\left\{ \left(\vert f(L,\ell,a)\vert + \vert f(\ell,L,a') \vert \right)  {e^{-(L+\ell)/5}}\right\}.$$

\begin{figure}
\centering
\subcaptionbox{}
{\includegraphics[width=0.52\textwidth]{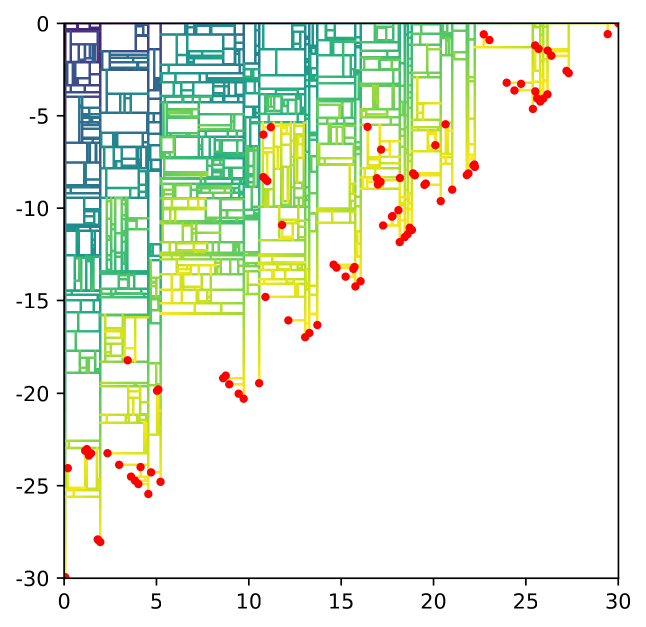}}
\subcaptionbox{}
{\includegraphics[width=0.21\textwidth]{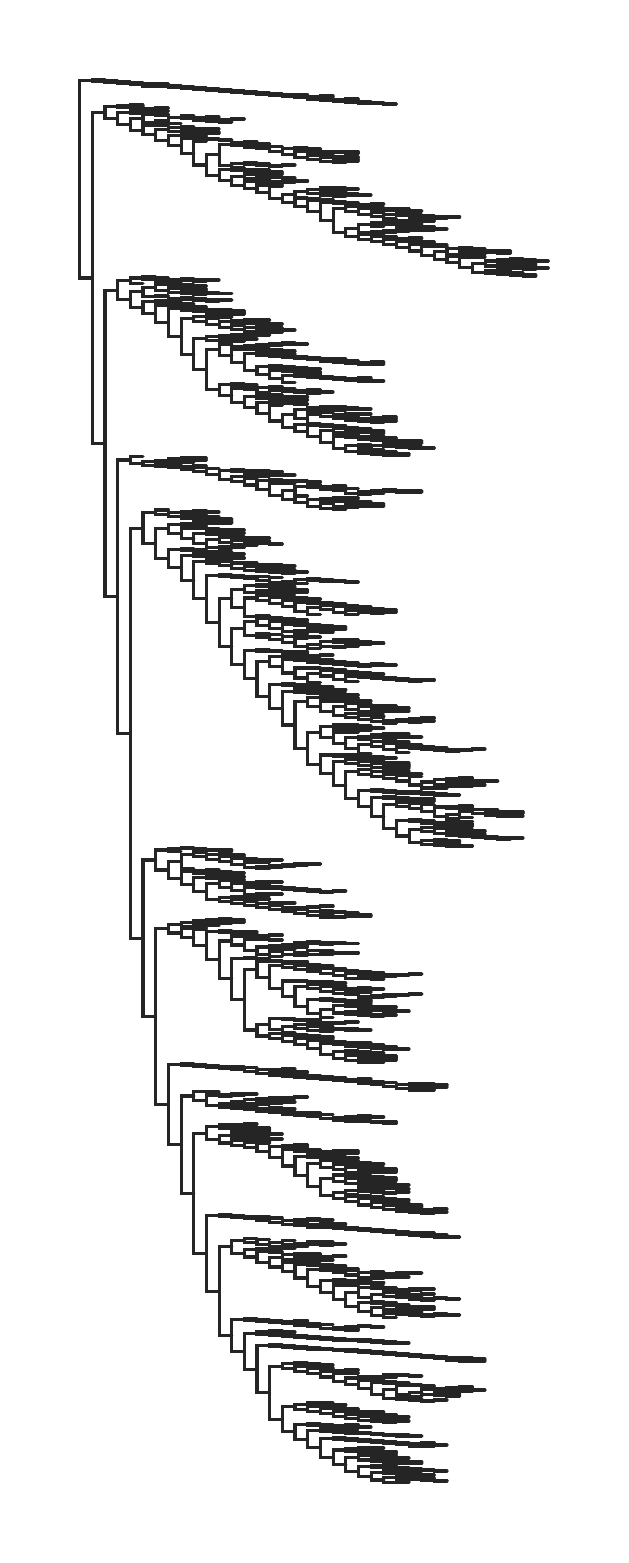}}
\caption{A simulation of the process: (a) the constructed network at $t=30$, (b) the corresponding genealogical tree.}\label{fig:simu1}
\end{figure}

\section{Construction of the network and  branching structure}
\label{constructdef}
The model we are interested in has been briefly described in the introduction. The purpose of this section is to build the corresponding planar stochastic process and link it to a branching process on the space of aging rectangles. First, in Section~\ref{subsec:skeleton} we construct a planar network with possible overlapping of its branches. Then in Section~\ref{subsec:fusion} we take into account the fusion events (inactivation of tips of branches when they hit an already constructed part of the network. Then we derive in Section~\ref{subsec:rectangles} the branching process describing the dynamics of the connected components of the network (rectangles). We conclude the loop by connecting everything to the model described in the introduction, see Section \ref{ConnectionMod}.

\subsection{Construction of the full skeleton   without  fusion}\label{subsec:skeleton}

We start by reminding the Ulam-Harris-Neveu notation for genealogy,
$$\mathcal U=\{(u_1, \ldots u_n) \, : \, n\in \mathbb N,\,  u_i \in \mathbb N\}.$$
An individual $u=(u_1, \cdots, u_n)\in \mathcal U$  will correspond to a branch, with $u_i \in \mathbb N$ and $n=\vert u \vert $ is the length of $u$.
We define two operations on $\mathcal U$ : for each $u=(u_1, \cdots, u_n)$, 
$$u^{\rightarrow} \, = \, (u_1, \ldots,u_{n-1}, u_n +1), \quad u^{\curvearrowright} \, = \, (u_1, \ldots, u_n,1),$$
which correspond to the two tips created from $u$ at branching,  resp. straight offspring  and the orthogonal offspring. 
We will also need  the ancestor of $u=(u_1,\ldots, u_{n-1}, u_n)$ corresponding to the last orthogonal branching, it is denoted by
$$u^{\perp}=(u_1,\ldots, u_{n-1}).$$
The direct ancestor $u^-$ of $u$ is
$$u^-=(u_1,\ldots, u_{n-1})=u^{\perp} \, \,  \text{if} \,  \, u_n=1; \quad
u^-=(u_1,\ldots, u_{n}-1) \, \,  \text{if} \, \,  u_n>1.$$
In particular, the first branch created  is labeled  by $u=(1)$ and it continues as $u^{\rightarrow}=(2)$ and its orthogonal offspring is $u^{\curvearrowright}=(11)$, whose orthogonal ancestor is also the direct ancestor : $(u^{\curvearrowright})^{\perp}=(u^{\curvearrowright})^-=(1)$. Note that the notations introduced in this Section are illustrated in Figure~\ref{fig:M01}. 

\begin{figure}
    \centering
    \includegraphics[width=0.5\textwidth]{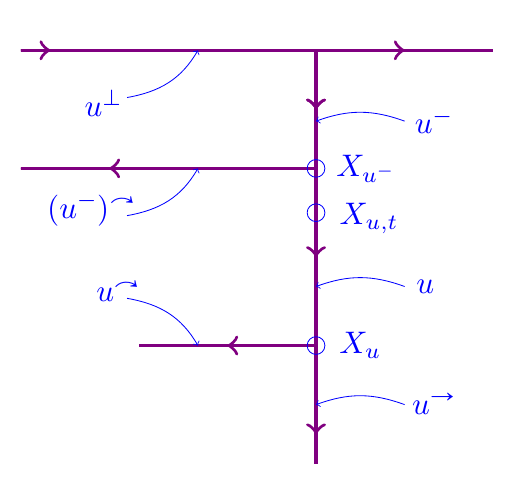}
    \caption{Notations introduced in Section~\ref{subsec:skeleton}}
    \label{fig:M01}
\end{figure}

Let us now define by induction the sequence of times and positions $(T_u,X_u)_{u\in \mathcal U}$ associated to different branches, without considering fusion events for now.  The first branching time, for the root $(1)$, is an exponential random variable $T_{(1)}$ with parameter $1$ and is also its length.
Independently of $T_{(1)}$, we  introduce    i.i.d. exponential random variables with parameter $1$,  denoted by  $(\Delta^{\rightarrow}_u, \Delta^{\curvearrowright}_u)$ for any  $u\in \mathcal U$. They give the length   of the two branches $u^{\rightarrow}$ 
and $u^{\curvearrowright}$.  We can then  define recursively the successive times of branching as follows: 
$$T_{u^{\rightarrow}}=T_u+\Delta^{\rightarrow}_u, \qquad T_{u^{\curvearrowright}}=T_u+\Delta^{\curvearrowright}_u$$
for any $u\in \mathcal U$.
For a given branch labeled by $u\in \mathcal U$, the time $T_u$ is its time of branching and $T_{u-}$ is its birth time, with convention $T_{(1)^-}=T_{\varnothing}=0$ a.s.\\ 
We can now define the corresponding positions in the plane, i.e. the position $X_u$ where $u$ branches, and $X_{u^-}$ the position where  $u$ is born (see Fig. \ref{fig:M01}). The branching events make the tips turn  on the right and after $4$ turns, we  are back in the original direction. That makes us  work with 
$\vert u \vert =4k +i$, $k\geq 0, \, i \in \{1,2,3, 4\}$: $i=1$ corresponds to branches $u$  growing to the right, $i=2$ to branches growing to the bottom, $i=3$ to the left and $i=4$ to the top.  The positions $X_u$ of the end of the branch $u$ can be defined inductively: for any 
$u\in \mathcal U$ such that $\vert u\vert \in 4\mathbb N_0 +1$,
$$X_{u^{\rightarrow}}=X_u+(\Delta^{\rightarrow}_u,0), \qquad X_{u^{\curvearrowright}}=X_u-(0,\Delta^{\curvearrowright}_u);$$
for any 
$u\in \mathcal U$ such that $\vert u\vert \in 4\mathbb N_0 +2$,
$$X_{u^{\rightarrow}}=X_u-(0,\Delta^{\rightarrow}_u), \qquad X_{u^{\curvearrowright}}=X_u-(\Delta^{\curvearrowright}_u,0);$$
for any 
$u\in \mathcal U$ such that $\vert u\vert \in 4\mathbb N_0 +3$,
$$X_{u^{\rightarrow}}=X_u-(\Delta^{\rightarrow}_u,0), \qquad X_{u^{\curvearrowright}}=X_u+(0,\Delta^{\curvearrowright}_u);$$
for any 
$u\in \mathcal U$ such that $\vert u\vert \in 4\mathbb N_0 +4$,
$$X_{u^{\rightarrow}}=X_u+(0,\Delta^{\rightarrow}_u), \qquad X_{u^{\curvearrowright}}=X_u+(\Delta^{\curvearrowright}_u,0).$$
We identify now the points corresponding to our growing network as elements
$(u,t)\in \mathcal U\times \R_+$, where $u$ gives  the branch to which  the point belongs   and $t$ the time of creation of the point.
Thus the set $$\mathcal R= \{ (u,t) : u\in \mathcal U, t\in \R_+,  T_{u^-}\leq t \leq T_u\}$$
gathers the elements which  are indeed created in the network (without fusion).  The location of   $(u,t)\in \mathcal U \times \R_+$  is
$$X_{u,t}=X_{u^-}\frac{T_u-t}{T_u-T_{u^-}}+X_u\frac{t-T_{u^-}}{T_u-T_{u^-}}.$$ 
Thus $\{ X_{u,t} : (u,t) \in \mathcal R \} \subset \R^2$ is the full network created by considering branching events but no fusion event.



\subsection{ Construction of the network with fusion}\label{subsec:fusions}
\label{branchfusion}
\label{subsec:fusion}
Until now we have constructed a network allowing overlapping. We are now ready to take into account the following fact:  When a tip of an elongating branch meets an existing branch, it should be inactivated. In other words, after the first overlapping, the further dynamics of the growing branch should be forgotten.  We will then see that the newly created network can be associated with a spatial random tree enjoying a branching property. We refer to Figure~\ref{fig:M02} for an illustration of the objects and notations introduced in this section.

\begin{figure}
    \centering
    \includegraphics[scale=0.9]{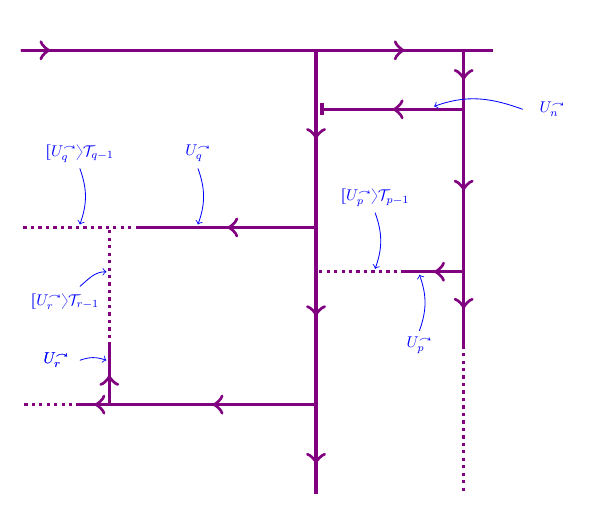}
    \caption{Illustration of the fusion events described in Section~\ref{subsec:fusion}. In this image, the network at time $t>0$ is represented in black, and the imprint is constituted of the black network, and the dotted lines. The fragment $U_n^\curvearrowright$ has already fused (thus $U_n^\curvearrowright\in W(t)$), while $U_{p}^\curvearrowright,U_{q}^\curvearrowright,U_{r}^\curvearrowright\in V(t)$ are active. We notice that the dotted lines are stopped by the fragments previously constituted (this is the case for $[U_{p}^\curvearrowright\rangle \mathcal T_{p-1}$, stopped by $u\in \mathcal T_{p-1}$), or by previously constituted dotted lines (as it is the case for $[U_{r}^\curvearrowright\rangle \mathcal T_{r-1}$, which is stopped by $[U_{q}^\curvearrowright\rangle \mathcal T_{q-1}$, with $q<r$).}
    \label{fig:M02}
\end{figure}

Recall that $(T_u)_{u\in \mathcal{U}}$ is the sequence of times at which $u$ branches. We now order the fragments by increasing birth time $(T_{u^-})_{u \in \mathcal U}$ and define  
$(U_n)_{n\in \mathbb N}$ so that a.s.
$$ \forall n\in \mathbb N, \, T_{U_n^-}\leq T_{U_{n+1}^-}, \qquad \{ U_n :  n \in \mathbb N\}=\mathcal U.$$
Notice that each branching event leads to the birth of two branches and
$$U_{2k}^-=U_{2k+1}^-, \qquad 
T_{U_{2k}^-}=T_{U_{2k+1}^-}$$ for $k\geq 1$. The branches $U_{2k}$ and $U_{2k+1}$ can be interchanged in this ordering. Let us choose first the branch that represents the straight offspring  and then the one that is the orthogonal offspring, i.e. we adopt the convention that $U_{2k}=(U_{2k}^-)^\rightarrow$, $U_{2k+1}=(U_{2k+1}^-)^\curvearrowright$.  
We write $[u\rangle$ for the half line  associated with $u$. It  starts from $X_{u-}$, with direction $X_u-X_{u^-}$ and is defined by
$$[u \rangle:=\{X_{u,t} : t\geq  T_{u^-} \}.$$
We note that $U_1=(1)$ and $[U_1\rangle=[0,\infty)\times \{0\}$ is the half line of $x$ coordinates.
For a halfline $[u\rangle$ and a given subset $\mathcal G$ of $\R^2$, we define the operation $[u\rangle \mathcal G$ of restricting the halfline to the part which has not encountered $\mathcal G$ as follows:
\begin{align*}
[u\rangle \mathcal G&:=\{ z\in [u\rangle \, : \,   (X_{u-},z] \cap \mathcal G=\varnothing \}=\{X_{u,t} : T_{u-}\leq t< T_{u}^{\infty}(\mathcal G)\}.
\end{align*}
Here $T_{u}^{\infty}(\mathcal G)\in [0,\infty]$ yields the time when the half line $[u \rangle$ encounters $\mathcal G$. When it indeed happens, i.e. $T_{u}^{\infty}(\mathcal G)<\infty$,  
we introduce also the point $X^{\infty}_u(\mathcal G)$ of intersection. It is characterized by $$X^{\infty}_u(\mathcal G)=X_{u,T_{u}^{\infty}(\mathcal G)} \text{ and } \quad [u\rangle \mathcal G=[X_{u-},X^{\infty}_u(\mathcal G)).$$
In words, 
the half line $[u\rangle$ generated by $u$ hits the set $\mathcal G$ for the first time in $X^{\infty}_u(\mathcal G)$ at time $T_{u}^{\infty}(\mathcal G)$.\\

We can now create inductively a spatial imprint $\mathcal T_n$  of the network  with fusion for $n\in\mathbb N$, such that
\[\mathcal T_0\subset\dots\subset\mathcal T_{n}\subset\mathcal T_{n+1}\subset\dots\subset\mathbb R^2.\]
We simultaneously define the times $T^{\infty}_u \in \R_+$ when each branch $u$ dies (by branching) or is inactivated  (by fusion).
We set 
$$\mathcal T_0:=(\{0\}\times (-\infty,0])\, \cup [U_1\rangle, \qquad T^{\infty}_{U_1}:=\infty,$$ 
i.e. $\mathcal T_0=(\{0\}\times (-\infty,0])\, \cup ([0,\infty)\times \{0\})$ is the boundary of the quadrant and the first branch $U_1=(1)$ grows forever.
For any
 $n\geq 1$, 
\begin{itemize}
\item If $X_{U_n}\in \mathcal T_{n-1}$, then $U_n$  branches   before inactivation. Hence, we  define \[T_{U_n^\curvearrowright}^\infty:=T_{U_n^\curvearrowright}^\infty\left(\mathcal T_{n-1}\right), \quad  T_{U_n^\rightarrow}^\infty:=T_{U_n}^\infty,\]
and 
we add  the half line corresponding to $U_n^\curvearrowright$  until it encounters the imprint $ \mathcal T_{n-1}$ to the imprint:
\begin{equation}\label{def:inprint}
\mathcal T_{n}:= \mathcal T_{n-1} \,  \cup \,[U_n^{\curvearrowright} \rangle \mathcal T_{n-1}.
\end{equation}
Notice that  $[U_n^\rightarrow\rangle$ already belongs to the imprint  $\mathcal T_{n-1}$, so  its time of life is the same as the one of $U_n$, while
$U_n^{\curvearrowright}$ stops when it encounters $\mathcal T_{n-1}$.
\item 
If $X_{U_n}\not\in \mathcal T_{n-1}$, then $U_n$ does not branch before inactivation. Therefore,
$$\mathcal T_{n}:=\mathcal T_{n-1},$$
and we do not need to define $T^{\infty}_{U_n^\rightarrow}$
 and $
T^{\infty}_{U_n^\curvearrowright}$.
\end{itemize}
This  sequence $(\mathcal T_n)_n$ adds up  the successive contributions of each branching event and its increasing limit provides the full imprint $\mathcal T$ of the network: 
$$\mathcal T=\cup_{n \in \mathbb N} \mathcal T_n.$$
Let $t>0$. We introduce the set of branches active at time $t$, i.e. the branches already born, that have not (yet) branched and that have not (yet) been inactivated:
\begin{equation}
\label{Vt}
V(t)=\{ u \in \mathcal U : T_{u-}\leq t < T_u\wedge T_u^\infty \}=\{ u \in \mathcal U : T_{u^-}\leq t < T_u\ \, , \,  X_{u,t} \in \mathcal T\}.
\end{equation}
Let also $W(t)$ the set of fragments that have been inactivated (by fusion) before time $t$:
\begin{equation}
\label{Wt}
W(t)=\{ u \in \cup_{s\in[0,t]}V(s) : T_u^\infty<T_u, \, T_u^\infty\leq t \} .
\end{equation}

Finally, the network $\mathcal N(t)$ at time $t$ can be defined as the collection of active points until time $t$ 
\begin{align*}
    &\mathcal N(t):=
\{ (u,s) \in \mathcal U \times [0,t] \, :  \,  T_{u-}\leq s < T_u\wedge T_u^\infty \}.
\end{align*}
Observe that each point in  the network    has  a unique representative $(u,t)$, except of the points where branching occurs. We also remark that $\mathcal N(t)$ is a subset of $\{ (u,s) \in \mathcal U \times [0,t]  : X  _{u,s} \in \mathcal T\}$. 

\subsection{A  branching process with freezing for aging rectangles}\label{subsec:rectangles}

The spatial network we constructed in the previous section  splits
the space (the quadrant) in a partition of rectangles, with possibly infinite sides,  where along some sides growth is occurring and, hence, they have not yet been completed. We  describe now the dynamics in time of this collection of rectangles which forms the random connected components of the network. The key point is that this splitting in disjoint rectangles  provides independence between them and a branching property. In other words, in this section,  we  reduce  the analysis of the  growing network to a branching process
of particles that are rectangles characterized by three real values, two for the length and width  and the third  for the age  that  represents the growing (elongation) along the length, while the width is completed. The growth will always happen along the first coordinate (representing the length variable).  
 \\

\begin{figure}
\centering
\subcaptionbox{}
{\includegraphics[scale=0.9]{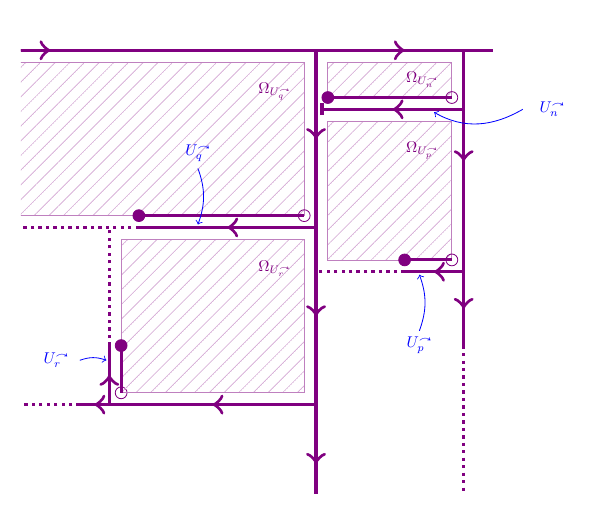}}
\subcaptionbox{}
{\includegraphics[scale=0.9]{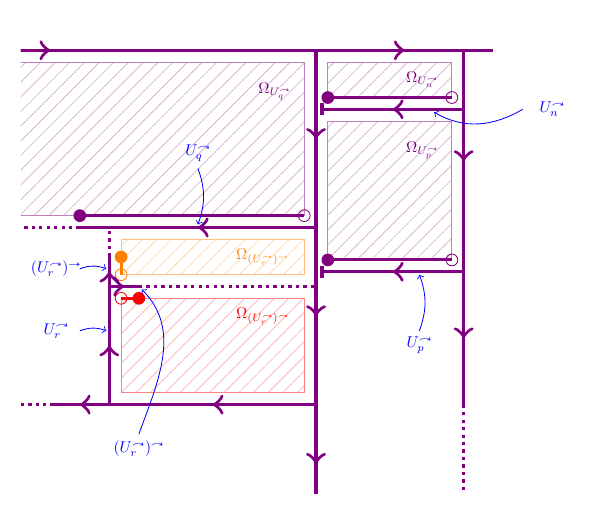}}
\caption{Illustration of the notations introduced in Section~\ref{subsec:rectangles}: (a) represents the network at a given time $t>0$, and (b) the same process at a later time $s>t$. The fragment $U_n^\rightarrow$ has already fused at time $t$ (i.e. $U_{n}^\curvearrowright\in W(t)$), and therefore this fragment, the associated rectangle $Z_{U_n^\curvearrowright}$ and its spatial representation  $\Omega_{U_n^\curvearrowright}$ are the same in (a) and (b). The fragment $U_{p}^\curvearrowright$ fuses between times $t$ and $s$, while no event (fusion or branching) is affecting $U_{q}^\curvearrowright$ during the time interval $[t,s]$. Finally, $U_{r}^\curvearrowright\in V(t)$ encounters a branching event at time $t$: the fragment $U_{r}^\curvearrowright$ then becomes inactive at time $s$, that is $U_{r}^\curvearrowright\notin V(s)\cup W(s)$, while the two descendent appear: $(U_{r}^\curvearrowright)^\rightarrow\in V(s)$, $(U_{r}^\curvearrowright)^\curvearrowright\in V(s)$. }\label{fig:sNotationsrectangles}
\end{figure}

For any $t\geq 0$, 
we recall that  $V(t)$ is the label set of branches $u$ active at time $t$, i.e.  which are born and have not yet branched or been inactivated, while $W(t)$ is the set of branches that have been inactivated before time $t$ (by fusion) and are frozen. 
To each branch $u\in V(t)\cup W(t)$ we associate a rectangle $Z_u(t)$, providing a partition of the quadrant. This rectangle  characterized by its length $L_u=T^{\infty}_{u}-T_{u^-}\in (0,+\infty]$  and  its width $\ell_u=T_{u^{\perp}}-T_{(u^{\perp})^-}\in (0,+\infty]$, which are fixed from birth. We also associate to $u$  its age $A_u(t)$ which corresponds to the time passed from the birth time $T_{u^-}$ , i.e. the time from which it actively grows. Thus,
\begin{align}
       \label{DefRec} 
Z_u(t)&=(L_u,\ell_u,A_u(t)) \nonumber \\
&:=(T^\infty_{u}-T_{u^-}, T_{u^{\perp}}-T_{(u^{\perp})^-},(t-T_{u^-})\wedge(T^\infty_{u}-T_{u^-}) ).
\end{align}
Note that this rectangle is not frozen (still active)  iff $0\leq A_u(t)<L_u$. It is then aging, i.e. growing at speed $1$ and belongs to $V(t)$.  It is created with $A_u(t)=0$, and frozen (inactivated) when $A_u(t)\geq  L_u$. 
We  notice that finite aging rectangles $Z_u(t)=(L_u,\ell_u,A_u(t))$  live in 
$\R_+^3$. Moreover the third coordinate is smaller than the first one in our construction. 
We also observe that the lengths  of rectangles on the boundary are infinite
 and one single fragment has length and width infinite. For convenience, we write $\mathcal X$ the full state space and set
$$\mathcal X= \{(L,\ell,a) \in \mathbb R_+^3 \, a\leq L\}  \, \cup \, (
\{+\infty\}\times \mathbb R_+^2 ) \,  \cup \,  (\{+\infty\}^2 \times \mathbb R_+).$$
Using \eqref{Vt} and \eqref{DefRec}, we introduce the random point measure on $\mathcal X \times \mathcal{U}$ which collects all 
the rectangles at time $t$, either active (and thus aging) or inactivated,
$$\overline{Z}_t=\sum_{u\in V(t) \cup W(t)} \delta_{(Z_u(t),u)}.$$
 Following for instance~\cite{BansayeMeleard2015}, the process  $\overline{Z}$ is represented by a stochastic differential equation using Poisson point measure. This yields in particular the following classical semi-martingale decomposition. At a branching time, if a rectangle has the coordinates $(L,\ell,a)$, we define the operations of fragmentation  
$\tau^\rightarrow$ and $\tau^\curvearrowright$  giving the  \textit{straight} and \textit{orthogonal}  offspring rectangles 
\begin{align}
\label{def:tau12}
\tau^\rightarrow(L,\ell,a)&=(L-a,\ell,0), \qquad
\tau^\curvearrowright(L,\ell,a)=(\ell,a,0)
\end{align}
and consider   functions $f : \mathcal X \times \mathcal U \rightarrow \R$ which are   continuous, bounded  and of class $C^1$  with respect to the third coordinate (age variable). 
Then $\overline{Z}$ satisfies for $t\geq 0$,
\begin{align}
&\langle \overline{Z}_t,f \rangle= \langle \overline{Z}_0,f \rangle + \int_0^t  \int_{\mathcal X \times \mathcal U}  \frac{\partial f}{\partial a} ( L,\ell,a, u) 1_{a<L} \overline{Z}_s (dL,d\ell,da,du) \,ds + M_t^f   \label{eq:semi-group}\\
&\, +\int_0^t  \int_{\mathcal X \times \mathcal U} 1_{a<L} \Big(f(\tau^\rightarrow(L,\ell,a),u^\rightarrow)
+f(\tau^\curvearrowright(L,\ell,a),u^\curvearrowright)  -f(L,\ell,a,u)\Big) \overline{Z}_s (dL,d\ell,da,du) \, ds,\nonumber 
\end{align}
where $(M_t^f)_{t\geq 0}$ is a martingale.
We also observe that by construction, the rectangles evolve independently  from one to another and  $\overline{Z}$ satisfies the branching property. Notice here that $\overline{Z}$ has been constructed in the previous section with the initial condition $\overline{Z}_0=\delta_{\infty, \infty,0, 1}$. This construction can be easily adapted to  an other initial condition $(L,\ell,a,1)\in \mathcal{X} \times \mathcal{U}$. \\
Let us finally observe that the branching process $\overline{Z}=(\sum_{u\in V(t) \cup W(t)} \delta_{(Z_u(t),u)} : t\geq 0)$ could be defined as the unique (weak) solution of 
\eqref{eq:semi-group} by using that 
\begin{equation}
    \label{def:V-rect}
    V(t)= \{u \in \mathcal{U}: \langle \overline{Z}_t, 1_u \rangle = 1 ,  0\leq A_u(t) < L_u\}, \quad 
    W(t)= \{u \in \mathcal{U}: \langle \overline{Z}_t, 1_u \rangle = 1 ,   A_u(t) =  L_u\}.
\end{equation}

\subsection{Connection to the original growing planar network model}
\label{ConnectionMod}


In this section, we explicit  the
connection between the process $\overline Z$ describing the family of rectangles and the spatial structure. We also a justify that it indeed yields the spatial network  described in the introduction. This section is not necessary for the rest of the paper, which analyses the branching process of rectangles and can be skipped. \\

First, it can be noticed that the positions $X_u$ can be obtained from the branching  process $\bar Z$, since $X_u$ is given by the rectangles associated to the ancestors of $u$.  Then, a rectangle $u\in V(t)\cup W(t)$ defined by \eqref{DefRec} can be  associated to a  (rectangular) set $\Omega_u$  of $\mathbb R^2$. More precisely, when $L_u,\ell_u<\infty$, we define
\begin{align}
    \Omega_u&:=\left\{X_{u^-}+\theta L_u\, \overrightarrow{X}_u+\eta \ell_u\textrm{ Rot}_{-\pi/2}\left( \overrightarrow{X}_u\right); \,  \,(\theta,\eta)\in[0,1]^2\right\},\label{def:Omegau}
\end{align}
where $\textrm{ Rot}_{-\pi/2}$ stands for the rotation of vector with angle $-\pi/2$ and 
\begin{equation}\label{def:vecZu}
    \overrightarrow{X}_u =\frac{X_{u}-X_{u^-}}{|X_{u}-X_{u^-}|}.
\end{equation}
The extension of the definition of $\Omega_u$ is straightforward  if $L_u=\infty$ and/or $\ell_u=\infty$. Then  $(\Omega_u)_{u\in V(t)\cup W(t)}$ covers the fourth quadrant in the following sense: 
\begin{prop}
For any $t\geq 0$,
\[
\cup_{u\in V(t)\cup W(t)} \Omega_u=\mathbb R_+\times \mathbb R_- , \qquad \textrm{Int }\Omega_u\cap \textrm{Int }\Omega_v=\emptyset\]
a.s. for any $u\ne v\in V(t)\cup W(t)$.
\end{prop}
\begin{figure}
\centering
\subcaptionbox{}
{\includegraphics[scale=0.9]{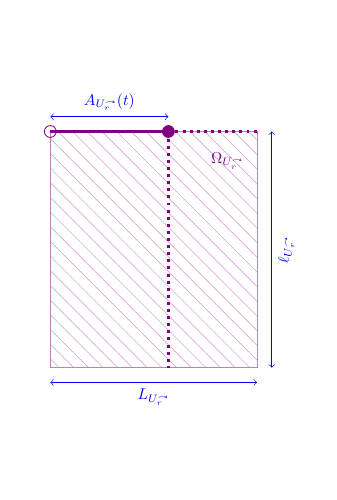}}
\subcaptionbox{}
{\includegraphics[scale=0.9]{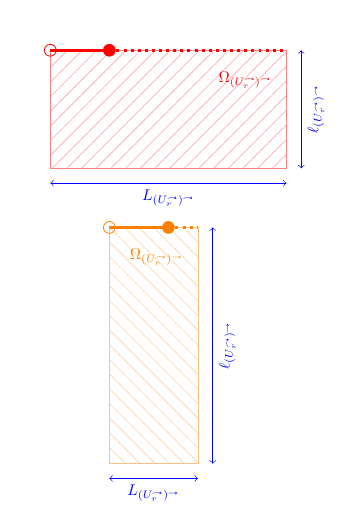}}
\caption{In (a), we represent the rectangles $U_{r}^\curvearrowright$ at time $t$ (see Figure~\ref{fig:sNotationsrectangles} (a)). In (b), we represent the two offspring at time $s>t$ of that rectangle, i.e. the rectangles $(U_{r}^\curvearrowright)^\curvearrowright$ and $(U_{r}^\curvearrowright)^\rightarrow$ (corresponding to Figure~\ref{fig:sNotationsrectangles} (b)).}\label{fig:sNotationsrectangles2}
\end{figure}

\begin{proof}
We use an inductive argument. Thanks to the definition of $(U_n)_{n\in\mathbb N}$, $V(t)\cup W(t)$ is constant for $t\in [T_{U_n},T_{U_{n+1}})$ for $n\in\mathbb N$: no branching event occurs in this time interval, and if a fusion event occurs for $t\in[T_{U_n},T_{U_{n+1}})$, the filament $u$ concerned with this fusion simply shifts from the set of fragments alive to the set of fragments that are fused, that is $V(t)=V(t^-)\setminus\{u\}$, $W(t)=W(t^-)\cup\{u\}$, and then $V(t)\cup W(t)=V(t^-)\cup W(t^-)$. \\
We notice that the desired property holds for $t\in[0,T_{U_1})$, since $V(t)\cup W(t)=\{U_1\}$,  $\Omega_1=\mathbb R_+\times \mathbb R_-$, and the statement of the proposition then holds for $t\in[0,T_{U_1})$. We make the induction assumption that the property holds for $t<T_{U_n}$, for $n\in\mathbb N$. 
\begin{itemize}
    \item If $U_n^-\notin V(T_{U_n}^-)\cup W(T_{U_n}^-)$, then  $U_n^-$ and its progeny are inactivated, which implies that $V(t)$ and $W(t)$ are constant for $t\in[T_{U_{n-1}},T_{U_{n+1}})$. 
    \item If $U_n^-\in W(T_{U_n}^-)$, then $U_n^-\in W(T_{U_n})$ and the progeny of $U_n^-$ are inactivated before birth, implying  that $V(t)$ and  $W(t)$ are constant for $t\in[T_{U_{n-1}},T_{U_{n+1}})$. 
    \item If $U_n^-\in V(T_{U_n}^-)$, then $(U_n)^-$ branches at time $T_{U_n}$ into $(U_n^-)^\rightarrow$ and $(U_n^-)^\curvearrowright$. The set $W(t)$ is   constant for $t\in[T_{U_{n-1}},T_{U_{n+1}})$, while $V(T_{U_n})=\left(V(T_{U_n}^-)\setminus\{(U_n)^-\}\right)\cup\{(U_n^-)^\rightarrow,(U_n^-)^\curvearrowright\}$. Moreover, from definition \eqref{def:Omegau} of $\Omega_u$,  
    \[\Omega_{(U_n)^-}=\Omega_{(U_n^-)^\rightarrow}\cup \Omega_{(U_n^-)^\curvearrowright},\quad \textrm{Int } \Omega_{(U_n^-)^\rightarrow}\cup\textrm{ Int } \Omega_{(U_n^-)^\curvearrowright}\subset \textrm{ Int }\Omega_{(U_n)^-},\]
    \[\textrm{Int } \Omega_{(U_n^-)^\rightarrow}\cap\textrm{ Int } \Omega_{(U_n^-)^\curvearrowright}=\emptyset.\]
  For illustrative argument, we refer to Figure~\ref{fig:sNotationsrectangles}(a)-(b), where the rectangle $\Omega_{U_r^\curvearrowright}$ (in (a)) branches into $\Omega_{(U_r^\curvearrowright)^\rightarrow}$ and $\Omega_{(U_r^\curvearrowright)^\curvearrowright}$ (in (b)). 
\end{itemize}
In these three cases, the statement of the proposition is then satisfied for $t< T_{U_{n+1}}$.
This ends the induction and the proof.
\end{proof}

\begin{figure}
\centering
\subcaptionbox{}
{\includegraphics[scale=1.1]{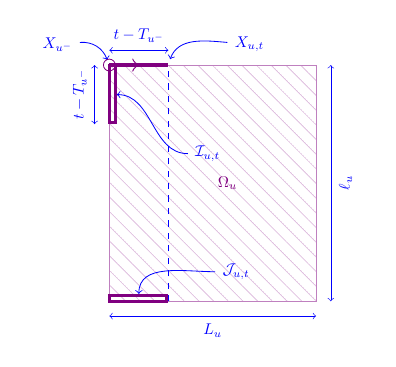}}
\subcaptionbox{}
{\includegraphics[scale=1.1]{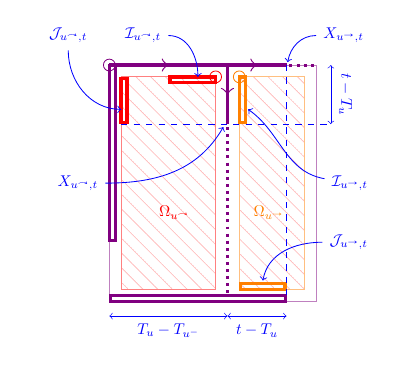}}
\caption{This figure illustrates the proof of Lemma~\ref{prop:originalmodel}. In (a), we represent $\mathcal I_{u,t}$ and $\mathcal J_{u,t}$ for $u\in V(t)$ and $t\geq T_{u^-}$: $u$ satisfies Property (P) if for any time $t\in [T_{u^-},T_u^\infty]$, branches cover the violet double lines $\mathcal I_{u,t}^1$, $\mathcal I_{u,t}^2$, and $\mathcal I_{u,t}^3$, that are part of the edges of rectangle $\Omega_u\subset\mathbb R^2$. In (b), we consider the situation after a branching has occurred at time $T_u$. We still represent in violet double lines the parts of $\mathcal I_{u,t}^2$ the segments covered by branches thanks to the condition (P) satisfied by $u$. The single thick violet line represents the branch generated by the tip $X_{u^\curvearrowright,t}$ that branched from $u$ at time $T_u$. In red (resp. orange) we represent the daughter rectangle $\Omega_{u^\curvearrowright}$ (resp. $\Omega_{u^\rightarrow}$) with a slightly reduced size to enhance the readability of the figure. We represent with double lines the segments that should be covered by branches for $u^\curvearrowright$ (resp. $u^\rightarrow$) to satisfy Property (P). We notice that the red and orange double lines are covered by the violet double lines and the single thick line, showing that if $u$ satisfies (P), then $u^\curvearrowright$ and  $u^\rightarrow$ also satisfy Property (P). }\label{fig:connectionoriginalmodel}
\end{figure}

Next, we show that for any $u\in \cup_{s\geq 0} V(s)$, the rectangle $\Omega_u$ has parts of its edges covered by branches created in the network at a given time. To state this lemma, we introduce the subset of $\mathbb R^2$ covered by network at time $t\geq 0$:
\begin{align*}
    \overline {\mathcal X}(t)&:=\{X_{u,s} \, : \,  (u,s)\in \mathcal N(t)\}\cup (\{0\}\times\mathbb R_-)\\
    &=\left\{X_{u^-}+(\sigma-T_{u^-})\vec X_u\, :\, u\in \cup_{s\geq 0}V(s),\, \sigma\in [T_{u^-},T_{u^-}+L_u]\cap[0,t]\right\}
    \cup \big(\{0\}\times\mathbb R_-\big).
\end{align*}
Notice that we add  $\{0\}\times\mathbb R_-$ to $\overline {\mathcal X}(t)$ to close the quarter plane where the network that we consider in this manuscript is developing. For $u\in\cup_{s\geq 0} V(s)$ and  $t\geq T_{u^-}$, we define two segments $\mathcal I_{u,t},\mathcal J_{u,t}\subset \partial \Omega_u$:
\[\begin{array}{l}
\mathcal I_{u,t}:=\left\{X_{u^-}+(s-T_{u^-})\textrm{Rot}_{-\pi/2}\left(\overrightarrow{X}_u\right); \, s\in[T_{u^-},t \wedge (T_{u^-}+\ell_u)]\right\},\\
\mathcal J_{u,t}:=\left\{X_{u^-}+\ell_u\textrm{Rot}_{-\pi/2}\left(\overrightarrow{X}_u\right)+(s-T_{u^-})\overrightarrow{X}_u;\,s\in[T_{u^-},t \wedge (T_{u^-}+L_u)]\right\},
\end{array}\]
We have represented these segments in violet in Figure~\ref{fig:connectionoriginalmodel} (a). The next lemma shows that the segments $\mathcal I_{u,t}$ and $\mathcal J_{u,t}$ are in the network (i.e. covered by branches) for any $t\geq T_{u^-}$:
\begin{lemme}\label{prop:originalmodel}
For $u\in\cup_{s\geq 0} V(s)$ and $t\geq T_{u^-}$,
\[\mathcal I_{u,t}\cup\mathcal J_{u,t}\subset \overline {\mathcal X}(t) \qquad \textsc{a.s}.\]
\end{lemme}

\begin{proof}
For $u\in \mathcal U$, we say that $u$ satisfies the property (P) if $\mathcal I_{u,t}$ and $\mathcal J_{u,t}$ are covered by branches for any $t\geq T_{u^-}$:
\begin{equation*}\label{def:PropP}
    \textrm{(P)}:\,\forall t\geq T_{u^-},\,\left\{\begin{array}{l}
\mathcal I_{u,t}\subset \overline  {\mathcal X}(t),\\
\mathcal J_{u,t}\subset \overline {\mathcal X}(t)\textrm{ if }\ell_u<\infty.
\end{array}\right.
\end{equation*}
Notice that if $\ell_u=\infty$, the segment $\mathcal J_{u,t}$ is located infinitely far, and we do not require that it is covered by branches (it will not be reached by the network anyways in finite time).

\medskip

We use an induction argument to show that property (P) is satisfied for any $u\in\cup_{s\geq 0} V(s)$, thanks to the fact that any $v\in\cup_{s\geq 0}V(s)$ is a descendent of $u=(1)$ through the effect of the operators ${}^\rightarrow$ and ${}^\curvearrowright$. To initiate this induction, we notice that $u=(1)$ satisfies the first line of property (P), since
\[X_{(0)}+(t-T_{u^-})\textrm{Rot}_{-\pi/2}\left(\overrightarrow{X}_{(1)}\right)=(0,0)+(t-T_{u^-})\textrm{Rot}_{-\pi/2}\left(\vec{e}_1\right)\subset\{0\}\times\mathbb R_+\subset\overline {\mathcal X}(t),\]
and the second line of property (P) also holds since $\ell_u=\infty$. Next, let $u\in\cup_{s\geq 0}V(s)$ satisfying (P). In $u$ does not undergoes a fusion event, it does not have descendent and the induction stops. We can therefore assume that $u$ does not undergoes a fusion event, ie such that $T_u<T_u^\infty$, so that we can define $u^\rightarrow\in\cup_{s\geq 0}V(s)$ and $u^\curvearrowright\in\cup_{s\geq 0}V(s)$. We show below that both $u^\rightarrow$ and $u^\curvearrowright$ satisfy condition (P), and we refer to Figure~\ref{fig:connectionoriginalmodel} for a graphic illustration of this induction argument:
\begin{itemize}
    \item For $u^\rightarrow$, we notice that for $t\geq T_u$,
    \begin{align*}
        \mathcal I_{u^\rightarrow,t}&=\left\{X_{u}+(s-T_{u})\textrm{Rot}_{-\pi/2}\left(\overrightarrow{X}_{u^\rightarrow}\right); \, s\in[T_{u},t \wedge (T_u+\ell_{u^\rightarrow})]\right\}\\
        &=\left\{X_{u}+(s-T_{u})\vec X_{u^\curvearrowright}; \, s\in[T_{u},t \wedge (T_u+L_{u^\curvearrowright})]\right\}\\
        &=\left\{X_{u^\curvearrowright,s}; \, s\in[T_{u},t \wedge T_{u^\curvearrowright}^\infty]\right\}\subset \overline {\mathcal X}(t),
    \end{align*}
    \begin{align*}
        \mathcal J_{u^\rightarrow,t}&=\left\{X_{u}+\ell_{u^\rightarrow}\textrm{Rot}_{-\pi/2}\left(\overrightarrow{X}_{u^\rightarrow}\right)+(s-T_{u})\overrightarrow{X}_{u^\rightarrow};\,s\in[T_{u},t \wedge (T_u+L_{u^\rightarrow})]\right\}\\
        &=\Big\{\left(X_{u^-}+(T_u-T_{u^-})\vec X_{u}\right)+\ell_{u^\rightarrow}\textrm{Rot}_{-\pi/2}\left(\overrightarrow{X}_{u^\rightarrow}\right)+(s-T_{u})\overrightarrow{X}_{u^\rightarrow};\\
        &\qquad \,s\in[T_{u},t \wedge (T_{u^-}+L_{u})]\Big\}\\
        &=\left\{X_{u^-}+\ell_{u^\rightarrow}\textrm{Rot}_{-\pi/2}\left(\overrightarrow{X}_{u}\right)+(s-T_{u^-})\overrightarrow{X}_{u};\,s\in[T_{u},t \wedge(T_{u^-}+ L_{u})]\right\}\subset \mathcal J_{u,t}\\
        &\subset \overline {\mathcal X}(t).
    \end{align*}
    Therefore $u^\rightarrow$ satisfies property (P).
    \item For $u^\curvearrowright$, we notice that  for $t\geq T_u$,
    \begin{align*}
        \mathcal I_{u^\curvearrowright,t}&=\left\{X_{u}+(s-T_{u})\textrm{Rot}_{-\pi/2}\left(\overrightarrow{X}_{u^\curvearrowright}\right); \, s\in[T_{u},t \wedge (T_{u}+\ell_{u^\curvearrowright})]\right\}\\
        &=\left\{X_{u}-(s-T_{u})\overrightarrow{X}_{u}; \, s\in[T_{u},t \wedge (T_{u}+(T_u-T_{u^-})]\right\}\\
        &\subset \left\{X_{u^-}+(s-T_{u^-})\overrightarrow{X}_{u}; \, s\in[T_{u^-},T_{u}]\right\}\subset \overline {\mathcal X}(T_u)\subset \overline {\mathcal X}(t),
    \end{align*}
    \begin{align*}
        \mathcal J_{u^\curvearrowright,t}&=\left\{X_{u}+\ell_{u^\curvearrowright}\textrm{Rot}_{-\pi/2}\left(\overrightarrow{X}_{u^\curvearrowright}\right)+(s-T_{u})\overrightarrow{X}_{u^\curvearrowright};\,s\in[T_{u},t \wedge (T_{u}+L_{u^\curvearrowright})]\right\}\\
        &=\left\{X_{u}-(T_u-T_{u^-})\overrightarrow{X}_{u}+(s-T_{u})\textrm{Rot}_{-\pi/2}\left(\overrightarrow{X}_{u}\right);\,s\in[T_{u},t \wedge (T_{u}+\ell_{u})]\right\}\\
        &=\left\{X_{u^-}+(\tilde s-T_{u^-})\textrm{Rot}_{-\pi/2}\left(\overrightarrow{X}_{u}\right);\,\tilde s\in[T_{u^-},(t-(T_u-T_{u^-})) \wedge (T_{u^-}+\ell_{u})]\right\}\\
        &\subset \mathcal I_{u,t-(T_u-T_{u^-})}\subset \overline {\mathcal X}(t-(T_u-T_{u^-}))\subset \overline {\mathcal X}(t).
    \end{align*}
    Therefore $u^\curvearrowright$ satisfies property (P).
\end{itemize}
This concludes the induction argument and the proof.

\end{proof}

In Section~\ref{subsec:fusion}, we described the half line $[u\rangle$ associated to any active tip $u\in\cup_{s\geq 0} V(s)$, and the imprint, which is the union of the branches already built and the half lines associated to each active tip (see \eqref{def:inprint}). In the proposition below, we show that for any active tip $u\in\cup_{s\geq 0} V(s)$, the point $X_{u,T_u^\infty}$ where the half line $[u\rangle$ encounters the imprint belongs to $\overline {\mathcal X}(t)$ for $t\geq T_{u^-}$. This shows that $X_{u,T_u^\infty}$ is part of  a branch for any $t\geq T_{u^-}$. In particular, any fusion event occurs at a point of the imprint that is already covered by branches at the time of the fusion.

\begin{prop}\label{Prop:fusion-non-imprint}
For any $t\geq 0$ and $u\in \mathcal U$, on the event
$\{u\in\cup_{s\geq 0} V(s), L_u<\infty,   T_{u^-} \leq t, T_u^\infty<\infty\}$,  $X_{u,T_u^\infty}$ belongs a.s. to   $\overline {\mathcal X}(t)$.
\end{prop}

\begin{figure}
\centering
{\includegraphics[scale=1.1]{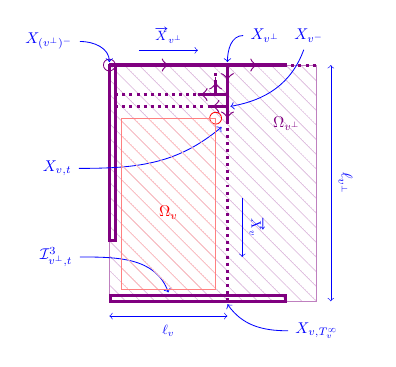}}
\caption{ In this picture we represent the fusion of the active tip $X_{v,t}$ with the edge of rectangle $\Omega_{v^\perp}$. We notice that  the point where the tip $X_{v,t}$ fuses belongs to $\mathcal J_{v^\perp,t}$, which is covered by branches at time $t$ since $v^\perp$ satisfies (P). }\label{fig:connectionoriginalmodel2}
\end{figure}

\begin{rque} Since fusions happen only when  an active tip encounters a branch already present, the the process we have introduced in Section~\ref{subsec:skeleton},~\ref{subsec:fusion} and \ref{subsec:rectangles} is actually identical to the growing planar network process described in the introduction.
\end{rque}

\begin{proof}
Let $\bar t> 0$ the time of a fusion event, and let $v$ the label of the active tip that is undergoing this fusion event. We define $u:=v^\perp$, and refer to Figure~\ref{fig:connectionoriginalmodel2} for an illustration of the notations used in this proof. The fusion is then happening at time $T_v^\infty$, at location 
\begin{align*}
    X_{v,T_v^\infty}&=X_{v^-}\frac{T_v-T_v^\infty}{T_v-T_{v^-}}+X_v\frac{T_v^\infty-T_{v^-}}{T_v-T_{v^-}}=X_{v^-}+\left(T_{v}^\infty-T_{v^-}\right)\frac{X_v-X_{v^-}}{|X_v-X_{v^-}|}\\
    &=\left(X_{v^\perp}+(T_{v^-}-T_{v^\perp})\overrightarrow{X}_v\right)+\left(T_{v}^\infty-T_{v^-}\right)\overrightarrow{X}_{v}\\
    &=X_{v^\perp}+\left(T_{v}^\infty-T_{v^\perp}\right)Rot_{-\pi/2}\left(\overrightarrow{X}_{v^\perp}\right)=X_{v^\perp}+\ell_{v^\perp}Rot_{-\pi/2}\left(\overrightarrow{X}_{v^\perp}\right).
\end{align*}
Therefore,
\begin{align*}
    X_{v,T_v^\infty}&=\left(X_{(v^\perp)^-}+(X_{v^\perp}-X_{(v^\perp)^-})\right)+\ell_{v^\perp}Rot_{-\pi/2}\left(\overrightarrow{X}_{v^\perp}\right)\\
    &=X_{(v^\perp)^-}+\ell_{v^\perp}Rot_{-\pi/2}\left(\overrightarrow{X}_{v^\perp}\right)+\left(T_{v^\perp}-T_{(v^\perp)^-}\right) \overrightarrow{X}_{v^\perp},
\end{align*}
which proves that $X_{v,T_v^\infty}\in \mathcal J_{v^\perp,T_{v^\perp}}\subset \overline {\mathcal X}(T_{v^\perp})$, thanks to  Lemma~\ref{prop:originalmodel}. Then, $X_{v,T_v^\infty}\in \overline {\mathcal X}(t)$ for any $t\geq T_{v^-}$. The property stated in the proposition then holds for any $v\in \cup_{s\geq 0} V(s)$ that undergoes a fusion event: $X_{v,T_v^\infty}\in \overline {\mathcal X}(T_{v^\perp})\subset \overline {\mathcal X}(T_{v^-})$. 

\medskip

We consider now $w\in\cup_{s\geq 0} V(s)$. There exist then $v=(\dots(u^\rightarrow)\dots)^\rightarrow$ that undergoes a fusion event, with $X_{v,T_v^\infty}=X_{u,T_u^\infty}$. We can define $u=v^\perp=w^\perp$, that satisfies $T_{v^\perp}\leq T_w$. Thanks to the first part of this proof, 
\[X_{w,T_w^\infty}=X_{v,T_v^\infty}\in \overline {\mathcal X}(T_{v^\perp})\subset \overline {\mathcal X}(T_{w}),\]
which completes the proof of the proposition.
\end{proof}



\section{Spinal construction for finite rectangles}
\label{spinal}

We recall that we denote by $Z$  the empirical measure of rectangles, i.e.  marginal on $\mathcal{X}$  of the process $\overline{Z}$ introduced in Section \ref{subsec:rectangles}:
$$Z_t=\sum_{u\in V(t)\cup W(t) } \delta_{(L_u, \ell_u, A_u(t))}.$$
In this section, we focus on finite rectangles and we study the first moment of the associated branching process, defined  by
$$M_tf(x)=\E_{\delta_z}(Z_t(f))$$
for any $t\geq 0$ and $z=(L,\ell,a)\in \R_+^3$ such that $a\leq L$ and $f : \mathbb R^3\rightarrow \R$ measurable and bounded. 
At the fragmentation events, the total surface of offspring is preserved and  we  consider the harmonic function giving the surface of rectangles $$h(z)=h(L,\ell,a)= \ell \times L$$ for $z=(L,\ell,a)\in \mathcal \R_+^3$. More precisely, $M_th=h$ for any $t\geq 0$ and proceed with the  (Doob) $h$  transform.
That corresponds to the classical size biased choice or tagged particle for  fragmentation processes, and we refer e.g. to the works of Bertoin and Haas and the monographs \cite{Bertoin, DJ}.
The dynamics of this typical rectangle is a Markov process 
$$Y_t=(L_t,\ell_t, A_t)$$
taking values in $\R_+^3$, whose    semigroup $P$ is  defined 
$$P_tf(z)=\frac{M_t (fh)(z)}{h(z)}=\E_z(f(Y_t)).
$$
for any $t\geq 0$ and $z=(L,\ell,a)\in \R_+^3$ such that $a\leq L$ and $f : \mathbb R^3\rightarrow \R$ measurable and bounded. Writing $G$ the generator of the semigroup $M$, 
the generator $\mathcal L$ of the Markov process $Y$  reads 
\begin{align}
\mathcal{L}f(z) &=\frac{G(hf)}{h}(z)=  1_{a< L} \left(\frac{L-a}{L}f(\tau^\rightarrow(z))+\frac{a}{L}f(\tau^\curvearrowright(z))-f(z)+ \frac{\partial f}{\partial a}(z)\right), \label{genetyp}
\end{align}
where $\tau^\rightarrow$ and $\tau^\curvearrowright$ are defined in \eqref{def:tau12} and $f$ is a bounded continuous function from $\R_+^3$ to $\R$ and  of class $C^1$ with respect to the third variable (age variable).
The process $Y$ represents the evolution of a typical rectangle in our population. Notice that when the age coordinate $A$ reaches $L$ the  length and width of $Y$ do not evolve anymore. When this happens (inactivation), as before we  say that $Y$ is \textit{frozen} and  the time of this event is
$$\tau:= \inf\{t\geq 0 : A_t\geq L_t\} \in \R_+ \qquad \text{a.s.}.$$
\subsection{Coupling  with  stick-breaking processes }
The spinal rectangle $Y$ constructed above lives in $\R_+^3$. We can 
decouple the two  spatial dimensions and reduce the process to a stick-breaking, i.e. a one dimensional aging fragment described by a Markov process living in $\R_+^2$. It is due in particular
to  the symmetry of the growth: growing to the right or the left does not change the fragmentation process and one can just wait that the again starts again along each dimension.  

More precisely,  we  construct the process $Y=(L_t,\ell_t, A_t)_{t\geq 0}$ starting 
from $(L, \ell,a)$
by considering two independent Markov  processes $Z^1=(L_t^1,A_t^1)_{t\geq 0}$ and $Z^2=(L_t^2,A_t^2)_{t\geq 0}$ in $\R_+^2$ with respective initial conditions $(L,a)$ and $(\ell,0)$ :
$$Z_0^1=(L,a), \qquad Z_0^2=(\ell,0), \qquad (0\leq a < L).$$
The processes $Z^1$ and $Z^2$ follow the same law: the age grows linearly until reaching the length, while  the length and age jump simultaneously at rate one with a size biased law for the length and a
reboot at $0$ for the age. More formally, the  generator $\mathcal A$ of $Z^1$ and $Z^2$ is given for any non negative numbers $a\leq L$ by
\begin{align}
\label{defZ1Z2gen}
    \mathcal{A}f(L,a)&
    = 1_{a< L} \left(\frac{\partial f}{\partial a}(L,a)+\frac{L-a}{L}f(0,L-a)+\frac{a}{L}f(0,a)-f(L,a)\right),
    \end{align}
where again $f$ is bounded continuous  on $\R_+^2$ and $C^1$ with respect to the second  variable (age). \\

Let us detail the coupling construction and refer to Figure \ref{fig:Y-Z1-Z2} for illustration. We construct $Y$ using the independent stick braking process $Z^1$ and $Z^2$ given above, by defining $Y$ on successive time intervals starting from $0$. First, 
$$Y_t=(L_t,\ell_t, A_t)=(L_t^1,L_0^2,A_t^1)$$
for $t<T_1$, where  $T_1$ is the first jump time of   $Z^1$ where $Z^1$ selects the left fragment at the stick breaking. It corresponds to a jump from  $Z^1_{t-}=(L_{t-}^1,A_{t-}^1)$
to $Z^1_{t}=(A_{t-}^1, 0)=(L_t^1, A_t^1)$ and can be a.s.  defined by
$$T_1:=\inf\{t \geq 0 :  L_t^1=A_{t-}^1\}$$
by noticing the branching a.s. occurs outside the middle of $[0,L_t^1]$.
Then the process evolves following the second dimension and for $t\in [0,T_2]$,
$$Y_{T_1+t}=(L_t^2, L_{T_1}^1, A_t^2),$$
where $T_2$ is the first time when $Z^2$ selects the left fragment 
$$T_2=\inf\{t \geq 0 :  L_t^2=A_{t-}^2\}.$$
At this time, we come back to the first dimension and repeat the procedure, using again $Z^1$, from the time $T_1$ we have left it. Repeating this procedure allows us to construct $Y$ on $\R_+$ from the paths of
$Z^1$ and $Z^2$. We can then  control the freezing time and the size of the rectangles from the corresponding quantities for the one dimensional fragments.
\begin{prop} Writing 
$$\tau_i=\inf\{t\geq 0 :  A_t^i\geq L_t^i\}$$ for $i\in \{1,2\}$, we have
\label{prop:YtoZatFreeze}
$$ \tau\leq \tau_1+ \tau_2, \quad L_{\tau} \geq L^1_{\tau_1}, \quad \ell_{\tau}\geq L^2_{\tau_2} \qquad \text{a.s.} $$
\end{prop}
\begin{proof}
We remark that the rectangle is frozen at time $\tau$ when the age reaches the boundary in one of the two dimensions. The two sizes of the rectangle stop simultaneously their evolution  at this time,   while $Z^1$ and $Z^2$ evolve (and freeze) independently. At time $\tau_1+\tau_2$, the two dimensions have frozen in our construction of $Y$ and   $\tau$ is smaller,  see Figure \ref{fig:Y-Z1-Z2}. For this reason, the sizes of the rectangle have less time to reduce than $(Z^1,Z^2)$, ensuring that $L_{\tau} \geq L^1_{\tau_1}$ and  $\ell_{\tau}\geq L^2_{\tau_2}$.
\end{proof}

\begin{figure}
    \centering
    \includegraphics[scale=3.7]{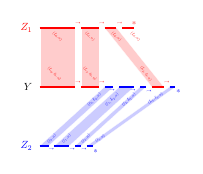}
    \caption{Construction of the process $Y$ from $Z_1$ and $Z_2$. Freezing are indicated by stars;   the choice of the right fragment  in the stick breaking is indicated by  $\rightarrow$ and the choice   left fragment by $\curvearrowright$.    }
    \label{fig:Y-Z1-Z2}
\end{figure}

\subsection{Freezing time and small fragments for stick process} 
We study the stick breaking process $Z^1=(L_t^1,A_t^1)_{t\geq 0}$ which appeared in the previous section. This Markov process on $\R_+^2$ is a fragmentation process with aging and freezing, defined by the generator $\mathcal{A}$ given in \eqref{defZ1Z2gen}. 
In words, the  age  of the stick grows at speed  $1$ and as long as it is smaller than the length of the stick, it splits at rate $1$. If a stick of size $L$ splits at age $a$, the new stick has length $L$ (and age $0$) with probability $a/L$ and otherwise it has length $L-a$ (and age $0$). 
We study the freezing time
$\tau_1$ of this stick process, which has been  defined by
$\tau_1=\inf\{t \geq 0 : L_t^1\geq A_t^1\}.$
We first give a large deviation estimate controlling the tail of this freezing time. It will be useful to quantify the speed of convergence of the spinal rectangle. We then control the size of the stick at freezing, which will be useful in particular to control the small fragments in the spinal rectangle.
\begin{prop} \label{Gel}
There exist $C_F,c_F>0$ such that for any $L\geq 0,a\in [0,L]$ and $t\geq 0$,
$$\Pp_{(L,a)}(\tau_1 \geq t)\leq C_F\exp\left(-c_F\frac{t}{L\vee 1}\right).$$
\end{prop}
Let us observe  before going to the proof  that this problem is related to  Gambler's ruin and the estimation of the  hitting time of zero for a random walk with negative drift starting from $L$. The difference here is that when we get close to zero, jumps may be notably big and close to $L$. We want here  estimates which hold for any initial length $L$ and will be used later.
\begin{proof} 
We first consider $L\leq 1$. We observe that with probability  at least $\exp(-L)\geq \exp(-1)$, no jumps occur before the age reaches the maximal value $L$  and freezing then occurs at time $L$. If a branching occurs before, we start again the process, with a length smaller than $L$. By iteration, we  get that the freezing time $\tau_1$ is stochastically smaller than $L$ times a geometric random variable with success probability $1-\exp(-1)$. It ensures that
$$\sup_{L\leq 1, a\leq L}\E_{(L,a)}(\exp(\lambda \tau_1))<\infty$$
for any $\lambda <1/(1-\exp(-1))$. Using Chernoff inequality  proves the proposition  for $L\leq 1 $.\\

Let us consider now  initial fragments bigger than $1$, i.e. $L\geq 1$. Let us deal first with initial  age $0$ for simplicity. The simple but key fact is that  
the probability to have no  point  between times $L/4$ and $L/2$  on a Poisson point process with parameter $1$ is $\exp(-(L/2-L/4))=\exp(-L/4)\leq \exp(-1/4)$. As a consequence, with probability at least $1-\exp(-1/4)$, a stick started from $L$ should  fragment between times $L/4$ and $L/2$. When this happens, the new stick has to be smaller than $3L/4$, whatever the choice at breaking. Observe that this  size reduction of factor $3/4$ still works when the stick already fragments before time $L/4$, whatever the choice of fragment at this breaking.   We get that
$$\alpha=\inf_{L\geq 1,a\leq L} \mathbb P_{(L,a)}\left(L^1_{L/2}\leq 3 L/4\right)\geq 1-\exp(-1/4) \in (0,1].$$
Iterating this argument, the time required to have a reduction of factor $3/4$ of the length $L$ is dominated by a geometric r.v. with success probability $\alpha$, multiplied by time $L$. This works until the length $L_t^1$ is below $1$ and ensures
that  starting from $(L,a)$, the freezing time $\tau_1$ satisfies
$$\tau_1\leq \sigma+L\sum_{i=0}^{\infty}   \text{Geom}_i(\alpha) (3/4)^i,$$
where $\text{Geom}_i(\alpha)$ are i.i.d. r.v.  with  geometric law with parameter $\alpha$ and $\sigma$
is an independent r.v. distributed as   the freezing time starting from length $1$  and age $0$.
We can control now the tail of the freezing time using classical Chernoff inequality  for $\lambda\geq 0$,
\begin{align*}
\mathbb P(\tau\geq t)&\leq e^{-\lambda t} \E\left(e^{\lambda\sigma}\right)
\Pi_{i\geq 0}
\E\left(\exp\left({\lambda L   (3/4)^{i}\text{Geom}_i(\alpha) }\right)\right).
\end{align*}
Choose $c \in (0,1-\exp(-1))$ such that $(1-\alpha)e^{c3/4}<1$ and  set $\lambda=c/L$.
Since $\lambda \leq c <(0,1-\exp(-1))), $
$$\E\left(e^{\lambda\sigma}\right)\leq \E\left(e^{c\sigma}\right)<\infty.$$
Moreover
$$\Pi_{i\geq 0} \E\left(\exp\left(\lambda L (3/4)^{i} \text{Geom}_i(\alpha)\right)\right)=\Pi_{i\geq 0}\frac{\alpha}{1-(1-\alpha)\exp({c(3/4)^i})},$$
which is finite and independent of $L$. This can be seen  by taking $\log$ of the product  and  using Taylor bounds of $\log$ and $\exp$ and $\log(1-x)$ and $\sum_{i\geq 0} (3/4)^i<\infty$. This concludes the proof.
\end{proof}

Let us  now  study the size of the frozen stick. The following bound  controls the tail of the size of the stick at freezing.  It will be required for  estimating the harmonic moment of the size of spinal rectangle at freezing time. 
\begin{prop}
\label{prop:oneDimControl}
For any non-negative measurable function $f$ and $L\geq 0$,
    \begin{equation}
    \label{eq:forf1}
    \E_{(L,0)} \left(\frac{f(L^1_{\tau_1})}{L^1_{\tau_1}}\right) \leq \frac{f(L)}{L } e^{-L} + 2e \|f\|_\infty .
\end{equation}
\end{prop}

Before we start the proof, we prepare it with the following lemma that will be also useful in further computations in the next section.

\begin{lemme}
\label{lemme:useful}
    For any non-negative measurable function $f$ and $L\geq 0$,
    \begin{align*}&\E_{(L,0)} \left(\frac{f(Z^1_{\tau_1})}{L^1_{\tau_1}}\right)
 =   \frac{f(L)}{L } e^{-L} + \frac{1}{L} M_f(L),
\end{align*}
where 
  $$M_f(x_0):=\sum_{k\geq 1}\int_{\R_+^{k}} {e^{-x_{k}}}  f(x_{k}) \Pi_{i=1}^{k}v_{x_{i-1}}(dx_{i})$$
  and the kernel $v$ on $\R_+$ is defined by
  $$v_{L}(dx)=1_{\{x<L\}} \left(e^{-(L-x)}  + e^{-x}\right) \ dx.$$
\end{lemme}
\begin{proof}
We use here the skeleton of the length component $(L^1_t)_{t\geq 0}$ of the stick. With a slight abuse of notation, we write $(L_i)_{0\leq i \leq N}$ the successive distinct values of this length process,    where $N$ is the number of jumps before freezing.
  This process starts from length $L_0=L$ at time $0$ and jumps from $L_{i}$ to $L_{i+1}$ for $i=1, \ldots, N$. Then it stays with length $L_N$ for ever. Observe that
  \begin{align*}
  \Pp(N\geq i+1, \, L_{i+1} \in dx_{i} \vert   N\geq i, L_i)
  &=\frac{x_i}{L_i} v_{L_i}(dx_i).
  \end{align*}
We get, for $k\geq 0$, writing $x_0=L$, by classical induction for Markov transitions
  \begin{align*}
 & \Pp_{L}(N= k, \, L_{i} \in dx_i : 1\leq i\leq k )\\
 &\qquad = \Pp_{L}(N= k \, \vert \, N\geq k, L_{i} \in dx_i : 1\leq i\leq k) \times
  \Pp_{L}(N\geq  k, \, L_{i} \in dx_i : 1\leq i\leq k)\\
  &\qquad =e^{-x_k} \times \Pi_{i=1}^{k}
  \Pp_{L}(N\geq i , \, L_{i} \in dx_i \vert N\geq i-1,  L_{i-1}=x_{i-1} )\\
&\qquad =e^{-x_k}\frac{x_k}{L}\Pi_{i=1}^k v_{x_{i-1}}(dx_i).
  \end{align*}
  We obtain
  $$\E_{(0,L_0)} \left[ 1_{ (N=k)} \frac{f(L_k)}{L_k} \right]=\frac{1}{L}\int_{\R_+^{k}} {e^{-x_{k}}}  f(x_{k}) \Pi_{i=1}^{k}v_{x_{i-1}}(dx_{i}) $$
  and we get
\begin{align*}&\E_{(L,0)} \left(\frac{f(Z^1_{\tau_1})}{L^1_{\tau_1}}\right)
=\sum_{k\geq 0} \E_{(L,0)} \left[ 1_{N=k} \,  \frac{f(L_k)}{L_k} \right]
    \label{eq:for2d}
 =   \frac{f(L)}{L } e^{-L} + \frac{1}{L} M_f(L),
\end{align*}
which ends the proof.
\end{proof}

\begin{proof}[Proof of Proposition \ref{prop:oneDimControl}]
 Applying Lemma \ref{lemme:useful}, we start from   \begin{align*}&\E_{(L,0)} \left(\frac{f(Z^1_\tau)}{L^1_\tau}\right)
 =   \frac{f(L)}{L } e^{-L} + \frac{1}{L} M_f(L).
\end{align*}
We now prove that $M_f$ grows (at most) linearly. For that purpose, we notice that bounding $f$ and the exponential by $1$ 
\begin{equation}
    \label{eq:estMwithU}
    M_f(x_0)\leq \|f\|_{\infty}\sum_{k\geq 1} u_k(x_0),   
\end{equation}
where  for $k\geq 1$,
\begin{equation}
\label{defuk}
u_k(x_0):=\int_{\R_+^{k}}  \Pi_{i=1}^{k}v_{x_{i-1}}(dx_{i}).
\end{equation}
We observe that setting $u_0(x)= 1$, the sequence $(u_n)_n$ satisfies  for $n\geq 0$
$$u_{n+1}(x)=\int_0^\infty  u_n(y) \, v_x(dy) =\int_0^x k(x,y) u_n(y) dy$$
with $k(x,y)= e^{-y}+ e^{-(x-y)} $. The function
\begin{equation}
\label{ladefdeu}
u(y)= \sum_{k\geq 0} u_k(y)    
\end{equation}
 is well defined by simply forgetting the exponentials and integrating over a simplex, i.e. the sum is convergent and we have $u(x)\leq 1+C e^{2x}$.
Multiplying  \eqref{ladefdeu} by $k(x,y)$ and integrating   w.r.t. to $dy$ on the interval $[0,x]$,  we obtain the relation
$$u(x)-1= \int_0^x k(x,y) u(y)dy. $$
Deriving two times this identity yields
 \begin{align*}
     u'(x)&= (1+e^{-x}) u(x)-e^{-x}\int_0^x e^y u(y) dy,\\
     u''(x)&=(1+e^{-x}) u'(x) - (1+e^{-x}) u(x) +e^{-x}\int_0^x e^y u(y) dy. 
 \end{align*}
 Summing the above, we obtain the ODE
 $$u''(x) = e^{-x} u'(x),$$
with $u(0)=1$ and $u'(0) = 2$.
This is easily solvable  and gives us that 
\begin{equation}
\label{borneu}
\sum_{k\geq 1}u_k(x)=u(x)-1= 2 \int_0^x e^{1-e^{-y}} dy \leq  2e x.
\end{equation}
Plugging this back to \eqref{eq:estMwithU} ends the proof of \eqref{eq:forf1}.
\end{proof}
\begin{rque} The estimates above will be enough for our purpose. 
    Inspecting the  proof, we 
    see that we can obtain a sharper estimate : for any non-negative measurable function $f$ and $L\geq 0$, 
    $$  \E_{(L,0)} \left(\frac{f(L^1_{\tau_1})}{L^1_{\tau_1}}\right) \leq \left(\frac{1}{L}+2e\right)\sup_{x\in \mathbb R_+} \vert  f(x)e^{-x}\vert.$$
    Such estimates will actually play a role in the forthcoming estimations involving infinite rectangles since these latter create finite rectangle with non bounded sizes.
\end{rque}

\subsection{Freezing time and small fragments for rectangles}

In this section,  we start from an initial rectangle which is finite and with age zero. We combine first our coupling  results with  stick-breaking processes  and 
the estimates obtained on them.
We derive an estimation of freezing time and a control of  the first moment of small rectangles. 
\begin{prop} \label{job}
For any  $L_, \ell<\infty$ and $a\in [0,L]$ and $t\geq 0$ and $\varepsilon\leq   L\ell$, we have
\begin{align}
   i) \quad  &\Pp_{(L, \ell,a)}(\tau \geq t)\leq 2C_F\exp\left(-\frac{c_Ft}{2(L\vee \ell \vee 1)}\right) \label{prop:smallrect} \\
    ii) \quad   &  M_t 1_{< \varepsilon}(L,\ell,a)\leq C\big( 1+L+\ell\big)^3, 
    \, \varepsilon_{\ell,L} 
    \end{align}
    where  $C_F,c_F$ are constants given by Proposition \ref{Gel} and  $1_{< \varepsilon}(x, y, b) = 1_{\{ x y < \varepsilon\}}$ and $C$ is also a universal constant (independent of the parameters) and 
    $$\varepsilon_{\ell,L}=\textsc{min}\left\{\varepsilon\left(  \frac{1}{\ell} +\frac{1}{L}+  \log\left(1+\frac{1}{\varepsilon}\right)\right) \, , \,  1 \right\}.$$
\end{prop}
\begin{proof}
The proof of $i)$ comes directly from Propositions \ref{Gel} and \ref{prop:YtoZatFreeze}, since
$$\mathbb  P(\tau \geq t)\leq \mathbb P(\tau_1\geq t/2)+\mathbb P(\tau_2\geq t/2).$$
The proof of $ii)$ requires more work.
We start from the $h$-transform,
    $$M_t 1_{< \varepsilon}(L,\ell,a) = h(L,\ell,a) P_t (1_{<\varepsilon}/h)(L,\ell,a),$$
where $P_t$ is the semigroup associated to the process $Y$ that represents the evolution of a typical rectangle in our population, see \eqref{genetyp} and the beginning of Section \ref{spinal}. 

Now, notice that $h$ is an increasing function of the surface of a given rectangle and that $1_{< \varepsilon}$ is decreasing with respect to the surface. As in its evolution $Y$ takes values such that the surface becomes smaller and smaller, for any $t>0$, we have that $$\frac{1_{< \varepsilon}}{h} (Y_t) \leq \frac{1_{< \varepsilon}}{h} (Y_\tau), $$ where $\tau$ is the freezing time of $Y$ and $Y_\tau$ is the state upon freezing. Hence, 
$$M_t 1_{< \varepsilon}(L,\ell,a) \leq h(L,\ell,a)\E_{(L,\ell,a)}\left[\frac{1_{<\varepsilon}}{h} (Y_\tau) \right]=L\ell \E_{(L,\ell,a)}\left[\frac{1_{L_\tau \ell_{\tau}< \varepsilon}}{L_\tau \ell_{\tau}}\right].$$
Now, we use the coupling with two independent realizations $Z^1$ and $Z^2$ of the two dimensional process $Z$ starting respectively from $(L,0)$ and $(\ell,0)$. In view of Proposition~\ref{prop:YtoZatFreeze} and the same monotonicity argument, we have that 
\begin{equation}
\label{eq:smallRecttoCoupling}
M_t 1_{<\varepsilon}(L,\ell,0) \leq Ll\E_{(L,0),(\ell,0)}\left[\frac{1_{ \{ L^1_{\tau_1} L^2_{\tau_2} < \varepsilon\} }}{L^1_{\tau_1} L^2_{\tau_2}} \right],
\end{equation}
where $L^1$ and $L^2$ are the lengths components of $Z^1$ and $Z^2$ respectively.

In what follows, we will focus on controlling the expectation on the right hand side of \eqref{eq:smallRecttoCoupling} using the independence of $Z^1$ and $Z^2$. We will do so in a two steps procedure. In the first step we condition w.r.t. $Z^2$ and perform the computations on $Z^1$ using Proposition~\ref{prop:oneDimControl}; In the second step we integrate w.r.t. $Z^2$ what has been obtained in  the first step.

{\bf First step.}
Conditioning w.r.t.  $Z^2$ we have
\begin{equation}
    \label{eq:step11}
   \E_{(L,a),(\ell,0)}\left[\frac{1_{ \{ L^1_{\tau_1} L^2_{\tau_2} < \varepsilon\} }}{L^1_{\tau_1} L^2_{\tau_2}} \right]=\E_{(\ell,0)}\left[\frac{1}{L^2_{\tau_2}}\E_{(L,a)} \left( \frac{1_{L^1_{\tau_1}<\varepsilon/L^2_{\tau_2}}}{L^1_{\tau_1}} \, \big| \,  Z^2\right) \right].
\end{equation}
Apply Lemma \ref{lemme:useful} for $f(z)=1_{z<\varepsilon/y}$ for a given $y$. It comes that
\begin{equation}
    \label{eq:step12}
  \E_{(L,0)} \Big[\frac{1_{L^1_{\tau_1<\varepsilon/y}}}{L^1_{\tau_1}} \Big] = \frac{1_{L<\varepsilon/y}}{L} e^{-L}+ \frac{1}{L} M_y(L).
\end{equation}
Here
$$z_0:= L \quad  \text{ and } \quad M_{\varepsilon,y}(L):=\sum_{k\geq 1}\int_{\R_+^{k}} {e^{-z_{k}}}  1_{z_k< \varepsilon/y} \Pi_{i=1}^{k}v_{z_{i-1}}(dz_{i}).$$
First, we integrate once w.r.t. the last variable ($dz_{k}$) bounding $e^{-z_k}v_{z_{k-1}}(dz_{k})\leq 2e^{-z_k}\times  1_{\{z_k\leq z_{k-1}\}}dz_{k}$. This gives
$$M_{\varepsilon,y}(L)\leq 2 \sum_{k\geq 1} \int_{\R_+^{k-1}} g_{\varepsilon,z_{k-1}}(y) 
\Pi_{i=1}^{k-1}v_{z_{i-1}}(dz_{i}),$$
where 
$$g_{\varepsilon, z}(y)= 1-e^{-(z \wedge \frac{\varepsilon}{y})}.$$
Due to the form of $v$, we have that $z_{k-1}\leq z_{k-2} \leq \dots \leq z_1\leq z_0=L $. Hence 
$$M_{\varepsilon,y}(L)\leq 2 g_{\varepsilon, L}(y)
\Big(1+\sum_{k\geq 2} \int_{\R_+^{k-1}}\Pi_{i=1}^{k-1}v_{z_{i-1}}(dz_{i}) \Big).$$
Performing a change in the index of summation, we obtain that 
$$M_{\varepsilon,y}(L)\leq 2g_{\varepsilon, L}(y)u(L),$$
where $u$ is defined in \eqref{ladefdeu}. Using \eqref{borneu} leads to
$$M_{\varepsilon,y}(L)\leq 2g_{\varepsilon, L}(y) (1+2eL).$$
Plugging the above in \eqref{eq:step12} 
yields 
\begin{equation}
    \label{eq:step13}
  \E_{(L,a)} \Big[\frac{1_{L^1_{\tau_1}<\varepsilon/y}}{L^1_{\tau_1}} \Big] \leq  \frac{1_{L<\varepsilon/y}}{L} e^{-L}+ \frac{2}{L} g_{\varepsilon, L}(y)(1+2eL).
\end{equation}
Inequality~\eqref{eq:step13} plugged in Eq.~\eqref{eq:step11} leads us to 
\begin{align}
    \label{eq:step21}
    &\E_{(L,a),(\ell,0)}\left[\frac{1_{ \{ L^1_{\tau_1} L^2_{\tau_2} < \varepsilon\} }}{L^1_{\tau_1} L^2_{\tau_2}}\right] \leq \,   I \, + \, II 
\end{align}
where 
$$I:=\frac{e^{-L}}{L}\E_{(\ell,0)}\Big(  \frac{1_{L<\varepsilon/L^2_{\tau_2}}}{L^2_{\tau_2}} \Big) $$
and 
$$II:=2\frac{ 1+2eL}{L}\E_{(\ell,0)}\Big( \frac{g_{\varepsilon, L}(L^2_{\tau_2})}{L^2_{\tau_2}}\Big).$$
{\bf Second step.} Noticing that 
$1_{L<\varepsilon/L^2_{\tau_2}}=1_{L^2_{\tau_2}<\varepsilon/L}$,
we can proceed for $I$ as in Step 1. If we change the role of  $(y,L^1_{\tau_1}, L)$ with $(L,L^2_{\tau_2}, \ell)$ in  \eqref{eq:step13} and using that $\varepsilon \leq L\ell$, we obtain
\begin{align}
    \label{eq:step22}
    I&
    \leq   2\frac{e^{-L}}{L\ell}  (1+2e\ell) (1-e^{- \frac{\varepsilon}{L}}).
\end{align}
It remains to treat $II$. 
As in Step 1, we apply Lemma \ref{lemme:useful} with $f=g_{\varepsilon, L}$ to obtain
$$\E_{(\ell,0)}\Big( \frac{g_{\varepsilon, L}(L^2_{\tau_2})}{L^2_{\tau_2}}\Big)= \frac{e^{-\ell}}{\ell} g_{\varepsilon, L} (\ell) +  \sum_{k\geq 1} \int_{\R_+^{k}} \frac{e^{-z_k}}{\ell} g_{\varepsilon, L} (z_k)\Pi_{i=1}^{k}v_{z_{i-1}}(dz_{i}). $$
Again, we bound $e^{-z_k}v_{z_{k-1}}(dz_{k})\leq 2e^{-z_k} \times 1\{z_k\leq z_{k-1}\}dz_k$ and use  that $z_{k-1}\leq \ell$. Denote 
\begin{align*}
b(L,\ell,\varepsilon):=\int_0^{\ell}e^{-x} (1-e^{-(L \wedge \frac{\varepsilon}{x})}) dx
\end{align*}
Hence, as before
\begin{align*}
    II
&\leq \frac{ 1+2eL}{L\ell}\Big(e^{-\ell} g_{\varepsilon, L} (\ell) + b(L, \ell, \varepsilon)( 1+2e\ell) \Big).
\end{align*} 
Gathering the estimates on $I$ and $II$  and plugging them in \eqref{eq:step21} yields   \begin{align*}
     \E_{(L,0),(\ell,0)}\left[\frac{1_{ \{ L^1_{\tau_1} L^2_{\tau_2} < \varepsilon\} }}{L^1_{\tau_1} L^2_{\tau_2}}\right]
    &\leq  2\frac{(1+2e\ell)}{L\ell}   (1-e^{- \frac{\varepsilon}{L}}) \\
    &\quad +  \frac{ 1+2eL}{L\ell}\Big(  1-e^{- \frac{\varepsilon}{\ell}} + b(L, \ell, \varepsilon)( 1+2e\ell) \Big).
\end{align*}
The above inequality plugged in \eqref{eq:smallRecttoCoupling} leads to 
\begin{align}
    \label{eq:conclusion}
    M_t f_\varepsilon(L,\ell,0) &\leq 2  (1+2e\ell) (1-e^{- \frac{\varepsilon}{L}}) \\
    & \qquad  +(1+2eL)\Big( 1-e^{- \frac{\varepsilon}{\ell}} + b(L, \ell, \varepsilon)( 1+2e\ell) \Big). \nonumber 
\end{align}
We use now 
\begin{align*}
b(L,\ell,\varepsilon)
&= 
(1-e^{-L})(1-e^{-\varepsilon/L})+\int_{\varepsilon/L}^{\ell}e^{-x}(1-e^{-\varepsilon/x}) dx 
 \end{align*}
 and 
 $0\leq 1-e^{-u}\leq 1\wedge u$ 
to get
\begin{align*}
\int_{\varepsilon/L}^{\ell}e^{-x}(1-e^{-\varepsilon/x}) dx &\leq 
\int_{\varepsilon/L}^{\ell\wedge \varepsilon}  e^{-x}dx+\int_{ \varepsilon}^{\infty} e^{-x}\frac{\varepsilon}{x}dx\\
&\qquad \leq \ell\wedge \varepsilon +(\varepsilon \wedge 1) (\log(1+1/\varepsilon)+1)
    \end{align*}
    and 
\begin{align*}
& M_tf_{\varepsilon} (L,\ell,0) \leq 2  (1+2e\ell)  (\frac{\varepsilon}{L}\wedge 1) \\
    & \quad   +(1+2eL)\Big( \frac{\varepsilon}{\ell} \wedge 1 +\Big(\frac{\varepsilon}{L} \wedge 1+ \ell(1\wedge \frac{\varepsilon }{\ell})+(\varepsilon \wedge 1)(\log(1+1/\varepsilon)+1)\Big) ( 1+2e\ell) \Big)
\end{align*}
It ends the proof.
\end{proof}

\section{First moment semigroup}
\label{Firstmomentsg}
In this section we obtain the convergence of the first moment semigroup of the branching process of rectangle
$$M_tf(z)=\E_{\delta_z}(Z_t(f))=\E_{\delta_z}\left(\sum_{u\in V(t)\cup W(t)} f(L_u, \ell_u, A_u(t)) \right),$$
where $z\in \mathcal X$ and $f : \mathcal X \rightarrow \mathbb R$ mesurable (bounded or non-negative).\\

First we treat finite initial condition $z\in \{(L,\ell,a)\in \R_+^3 : a\leq L\}$ and then combine it with  double Poisson immigration structure to obtain the desired result also for infinite initial rectangles (with one or two infinite sides). 

\subsection{Quantitative estimates for  finite rectangles}
We now control the convergence of the first moment semigroup
$$M_tf(z)=h(z)P_t(f/h)$$
to the stationary distribution $h(z)\pi_z(f/h)$, where $z\in \{(L,\ell,a)\in \R_+^3 : a\leq L\}$ $\pi_z$ is the  limiting law of the spinal rectangle $Y$ starting from $z=(L,\ell,a)\in \R_+^3$, i.e. its value at freezing  :
\begin{equation} 
\label{defpiz}
\pi_z(f)=\E_z(f(Y_{\tau}))=\E_z(f(L_{\tau}, \ell_{\tau}, A_{\tau})), \quad \tau:= \inf\{t\geq 0 : A_t\geq L_t\}.
\end{equation}
\begin{prop} \label{job2}
For any  $z=(L,\ell,a)\in \R_+^3$ such that  $a\in [0,L]$, for any $t\geq 0$  and $f$ non-negative measurable bounded, 
\begin{align}
   i) \quad  & M_t f (z)\leq C\parallel f \parallel_{\infty}\big( 1+L+\ell\big)^3 
\label{eq:controlSmallRect} \\
    ii) \quad   &
    | M_t f(z) - h(z) \pi_{z}(f/h)|\leq C_{\alpha} \parallel f \parallel_{\infty}\frac{(1+L+\ell  )^{4\vee 3\alpha}}{(t+1)^{\alpha}}, \label{ctrfini}
    \end{align}
    for any $\alpha>1$, 
    where  the constant $C$ (resp. $C_{\alpha}$) is universal (resp. only depends on $\alpha$).
\end{prop}
This provides  (arbitrary) polynomial speed of convergence, with quantitative control in terms of  the initial size $(L,\ell)$. This will  will be useful   for forthcoming  convergences, including $L^2$ convergence of the next section.
\begin{proof} For $z=(L,\ell,a)\in \R_+^3$ and $t\geq 0$ and $f$ non negative and bounded,
we first observe that
$$M_t(f1_{\geq L\ell})(z)\leq \parallel f \parallel_{\infty}M_t1_{\geq L\ell}(z)\leq \parallel f \parallel_{\infty}e^{-t\wedge (L-a)}.$$
Indeed, $M_t1_{\geq L\ell}(z)$ evaluates the probability that the initial rectangle $(L,\ell,a)$ has not fragmented yet at time $t$, when $t\leq L-a$, while it is frozen after time $L-a$.
Writing $f=f1_{\geq L\ell}+f1_{<L\ell}$, the first part $i)$ is then a consequence of Proposition \ref{job} $ii)$ for $\varepsilon=L\ell$. 

For the convergence, we need to take care of small fragments, since we divide by $h$ which vanishes at $0$. We first consider
\begin{align*}
\vert M_t(f1_{\geq \varepsilon)}(z)-h(z)\pi_z(f1_{\geq \varepsilon}/h) \vert &=
h(z) \vert \E_z(f1_{\geq \varepsilon}/h (Y_t))-\E_z(f1_{\geq \varepsilon}/h(Y_{\tau})) \vert\\
&\leq h(z) \parallel f1_{\geq \varepsilon }/h \parallel_{\infty}\mathbb P_z(\tau \geq t).
\end{align*}
Letting $t$ goes to infinity, the right hand side goes to $0$ by  Proposition \ref{job} $i)$. 
Using  then  Proposition \ref{job2} $i)$ for the left hand side and letting $\varepsilon \downarrow 0$ yields by monotonicity 
\begin{equation}
\label{massez}
h(z)\pi_z(f/h)=\lim_{\varepsilon \rightarrow 0} h(z)\pi_z(f1_{\varepsilon}/h)
 \leq C\parallel f \parallel_{\infty}\big( 1+L+\ell\big)^3.  
\end{equation}
Similarly, for any $\varepsilon>0$,
\begin{align*}
h(z)\pi_z(1_{<\varepsilon}/h) & 
=\lim_{\eta\rightarrow 0} h(z)\pi_z(1_{\eta\leq . <\varepsilon}/h)\\
& =\lim_{\eta\rightarrow 0} \limsup_{t\rightarrow \infty} M_t1_{\eta\leq . <\varepsilon}\leq C\big( 1+L+\ell\big)^3 
    \varepsilon_{\ell,L}
    \end{align*}
    by using now Proposition \ref{job} $ii)$.
Let us proceed similarly with $0\leq f=f1_{<\varepsilon}+f1_{\geq\varepsilon} \leq 1$ and use
\begin{align*}
\vert M_tf(z)-h(z)\pi_z(f/h) \vert &\leq h(z)\varepsilon^{-1}\mathbb P_z(\tau \geq t)\\
&\qquad +M_t(1_{< \varepsilon})+h(z) \pi_z(1_{< \varepsilon}/h).
\end{align*}
Using respectively Proposition \ref{job} $i)$ for the first term and   Proposition \ref{job} $ii)$ for the second term and the previous inequality for $h(z) \pi_z(1_{< \varepsilon}/h) $, we get
\begin{align*}
\vert M_tf(z)-h(z)\pi_z(f/h) \vert &\leq 2C_F L \ell\varepsilon^{-1}
\exp\left(-\frac{c_Ft}{2(L\vee \ell \vee 1)}\right) +2 C\big( 1+\ell+L\big)^3 \varepsilon_{\ell,L}.
\end{align*}
We choose
$$\varepsilon=\frac{L\ell }{(t+1)^{2\alpha}}$$
for $\alpha>1$  fixed. It is smaller than $L\ell$ and using also that $\varepsilon \log(1+1/\varepsilon)$ is dominated by $\sqrt{\varepsilon}$, we get for  any $L,\ell,t$ in $\R_+$,
\begin{align}
 | M_t f(z) - h(z) \pi_z(f/h)| &\leq L l \varepsilon^{-1}  \exp\left(-\frac{c_F}{2(L\vee \ell \vee 1)}t\right) +
 \frac{C'\big( 1+\ell+L\big)^4}{(t+1)^{\alpha}}. \nonumber 
 \end{align}
Using that $(1+x)^{3\alpha}e^{-x}$ is bounded on $\R_+$ by $C_{\alpha}$ and $x=c_Ft/2(L\vee \ell \vee 1)$, we obtain that  
$$L \ell \varepsilon^{-1}  \exp\left(-\frac{c_F}{2(L\vee \ell \vee 1)}t\right) \leq C_{\alpha}\frac{(t+1)^{2\alpha}}{(1+x)^{3\alpha}}\leq C_{\alpha}'\frac{(1+L\vee \ell \vee 1)^{3\alpha}}{(t+1)^{\alpha}},$$ which proves  $ii)$ for any $f$.
\end{proof}

\subsection{Infinite rectangles and immigration}
We can now fully describe the first moment semigroup, by using that it  is a branching structure with immigration of finite rectangles and plug the estimates obtain for the latter. \\

Let us  start from one rectangle with one infinite side and consider $M_tf(z_{\ell})$ for $z_{\ell}=(\infty, \ell,0)$. When this initial rectangle fragments, it creates a  rectangle with both length and width finite (orthogonal offspring) and a rectangle with one infinite side of the same type $(\infty, \ell,0)$ but different label (straight offspring).  These two rectangles then evolve independently by branching property.  The rectangle with infinite side then does the same thing again and again, creating a sequence of times at which finite rectangles are created from $z_{\ell}=(\infty, \ell,0)$. We will denote this sequence 
by $(T_i)_{i\geq 1}$ which follows a Poisson point process on $\R_+$ with parameter $1$. Writing $T_0=0$,  these successive divisions hence yield a finite rectangle $(\ell,T_{i}-T_{i-1},0)$ and an infinite rectangle $(\infty,\ell,0)$ at time $T_i$ for $i\geq 1$. 
 At any time there is  only one  rectangle with  infinite length  and  its width is always equal to $\ell$. Let us focus on   on finite rectangles produced by  the initial rectangle $\delta_{z_{\ell}}$. We thus    take $f$ supported by $\R_+^3$ and decompose the population of finite rectangles depending on the time $T_i \leq t$ where their ancestor on the $x$ coordinate has branched. We   get for any $t\geq 0$,
\begin{align*}
M_tf(z_{\ell})&=\sum_{i\geq 1} \E\left(1_{T_i\leq t} M_{t-T_i}f(\ell,T_{i}-T_{i-1},0)\right)\\
&=\sum_{i\geq 1} \E\left(1_{ T_{i}-T_{i-1}\leq t-T_{i-1}} M_{t-T_{i-1}-(T_{i}-T_{i-1})}f(\ell,T_{i}-T_{i-1},0)\right)\\
&=\sum_{i\geq 1} \E\left(F_{\ell}(t-T_{i-1})\right),
\end{align*}
where for $v\in \mathbb R$, we define
$$F_\ell(v)= \E(1_{ \Delta\leq v}M_{v-\Delta}f(\ell,\Delta,0)),$$
with $\Delta$ exponential random variable with parameter one.
We get
\begin{align}
\label{expressimm}
M_tf(z_\ell)&=F_\ell(t)+\int_0^tF_\ell(t-s)ds.
\end{align}
We start  from infinite rectangle $z_{\ell}$ and the lengths of  finite rectangles created are no longer bounded. However, the sizes of these finite rectangles will be dominated by exponential distribution, due to the Poisson arrival of branching events. This makes the following norm relevant for the forthcoming estimations
\begin{equation}
\label{normel}
\parallel f \parallel_{\ell,\infty}=\sup_{\substack{(x,y,a,a')\in \mathbb R_+^4 \\  a\leq x,\,  \,  a'\leq y\leq \ell}} (\vert f(x,y,a)\vert +\vert f(y,x,a')\vert)e^{-3x/4}.
\end{equation}
We can now describe the first moment and prove that the asymptotic profile is given by
\begin{equation}
\label{pibardef}
\Pi_{\ell}(f)=\ell\int_0^{\infty} e^{-x} x \pi_{(\ell,x,0)}(f/h)dx
\end{equation}
and we recall that $\pi_{(\ell,x,0)}$ has been defined in \eqref{defpiz}.
\begin{prop} \label{propinftyl} For any $\ell\in \R_+$ and $t\geq 0$ and  $f$ measurable and with support in $\R_+^3$ such that $\parallel f \parallel_{\ell,\infty}<\infty$, we have 
$$M_tf(\infty, \ell,0) \leq C\parallel f \parallel_{\ell,\infty}(1+t)(1+\ell)^3.$$  
and 
  $$ \big\vert M_tf(\infty,\ell,0) -t\Pi_\ell(f)\big\vert\leq  C_{\alpha}\parallel f \parallel_{\ell,\infty}(1+\ell)^{4\vee 3\alpha},$$
for any $\alpha>1$,  where $C$  (resp. $C_{\alpha}$) is a (universal) constant  (resp. only depends on $\alpha$). 
   \end{prop}
As expected from the immigration structure, the mass starting from  one infinite rectangle grows linearly and the renormalized semigroup converges at polynomial speed. The explicit control with respect to the initial width of the rectangle will be useful below.\\

For the proof, we first estimate $F_{\ell}$ using the previous results on finite rectangles for the first moment $M$ and we then derive the asymptotic behavior of $M_tf(\infty, \ell, 0)$ for $t$ large.
\begin{lemme}
\label{estilem}
There exists $C>0$ such that
for any  $f$  measurable function $f : \mathcal \rightarrow \R_+$ with support in $\R_+^3$ such that $\parallel f \parallel_{\ell,\infty} <\infty$  and any $l\geq 0$  and $v\geq 0$, and any $\alpha>1$,
$$\vert F_\ell(v)-\Pi_\ell(f)\vert\leq  C\parallel f \parallel_{\ell,\infty}\frac{(1+\ell)^{4\vee 3\alpha}}{ (1+v)^{\alpha}},$$
 where we recall that $\Pi_\ell$ is defined in \eqref{pibardef} and $\parallel . \parallel_{\ell,\infty}$ in \eqref{normel}.
Moreover, 
$$F_{\ell}(v)\leq C\parallel f \parallel_{\ell,\infty}(1+\ell)^3.$$
\end{lemme}
\begin{proof}
For $v\geq 0$,
\begin{align*}
F_{\ell}(v)&=\int_0^v e^{-x}M_{v-x}f(\ell,x,0)dx\\
&=\int_{0}^{v}e^{-x} \ell x \pi_{(\ell,x,0)}(f/h)dx+
\int_0^v e^{-x}(M_{v-x}f(\ell,x,0)-\ell x \pi_{(\ell,x,0)}(f/h))dx 
\end{align*}
and
\begin{align*}
\vert F_{\ell}(v)-\Pi_\ell(f)\vert 
&\leq \int_{v}^{\infty} e^{-x} \ell x \pi_{(\ell,x,0)}(f/h)dx +\bigg\vert \int_0^v e^{-x}(M_{v-x}f(\ell,x,0)-\ell x \pi_{(\ell,x,0)}(f/h))dx\bigg\vert. 
\end{align*}
We use  \eqref{ctrfini} for the last term and the fact that finite rectangles only decrease. More precisely,  we  consider the possible states starting from a rectangle $(\ell,x,0)$, given by
\begin{align*}
    E_{\ell,x}&=\{(y,z,a) \in \R_+^3 \,  : \,  0\leq a\leq y \leq \ell, 0\leq z\leq x \}\\
    &\qquad \cup \,  \{(y,z,a)\in \R_+^3 : 0\leq a\leq y \leq x, 0\leq z \leq \ell \}.
 \end{align*}
Hence, $M_{v-x}f(\ell,x,0)=M_{v-x}(f1_{E_{\ell,x}})(\ell,x,0)$
and \eqref{ctrfini} yields
 $$\left\vert M_{v-x}f(\ell,x,0)-\ell x \pi_{(\ell,x,0)}(f/h)\right\vert \leq C_{\alpha}\sup_{E_{\ell,x}} \vert f \vert \frac{(1+\ell+x  )^{4\vee 3\alpha}}{(v-x+1)^{\alpha}}$$
 and integrating this inequality we get
\begin{align*}
    &\big\vert \int_0^v e^{-x}(M_{v-x}f(\ell,x,0)-\ell x \pi_{(\ell,x,0)}(f/h))dx \big\vert\\
    &\qquad \leq C_{\alpha} \parallel f \parallel_{\ell,\infty} \int_0^v e^{-x/4} \frac{(1+\ell+x  )^{4\vee 3\alpha}}{(v-x+1)^{\alpha}} dx \leq  C_{\alpha}' \parallel f \parallel_{\ell,\infty} \frac{(1+\ell)^{4\vee 3\alpha}}{ (1+v)^{\alpha}}
    \end{align*}
    since $v-x+1\geq (1+v)/(1+x)$ for $x\in [0,v]$. 
    Using \eqref{massez}, we also have
\begin{align*}
\int_{v}^{\infty}  e^{-x}\ell x \pi_{(\ell,x,0)}(f/h)dx&\leq C\parallel f \parallel_{\ell,\infty} \int_{v}^{\infty}  e^{-x/4}(1+x+\ell)^3dx \\
& \leq C \parallel f \parallel_{\ell,\infty} e^{-v/4}(1+v)^3(1+\ell)^3 \int_0^{\infty} e^{-u/4}\left(1+\frac{u}{1+v}\right)^3 du\\
&\leq C'\parallel f \parallel_{\ell,\infty} e^{-v/4}(1+v)^3(1+\ell)^3.
\end{align*}
Combining these results and using that $\exp(-v/4)(1+v)^3$ is dominated by $1/(1+v)^{\alpha}$ yields the first part of the lemma.
Similarly
$$M_{s}f(\ell,x,0)\leq
\parallel f \parallel_{\ell,\infty} e^{x/4} M_{s}1(\ell,x,0)\leq C \parallel f \parallel_{\ell,\infty} e^{x/4} (1+\ell+x)^3 $$
from Proposition \ref{job2} $i)$.  
Then, integrating this estimate along time,  
\begin{align*}
F_\ell(v)&=\int_{0}^v e^{-x}M_{v-x}f(\ell,x,0)dx\leq C'\parallel f \parallel_{\ell,\infty}(1+\ell)^3,
\end{align*}
 which ends the proof.
\end{proof}
Combining   \eqref{expressimm} and the estimates of Lemma \ref{estilem} proves Proposition \ref{pibardef}.

\subsection{Stationary profile}
Let us now describe 
the first moment semigroup starting from the rectangle whose length and width are infinite, i.e. $M_tf(z_{\infty})$ for $z_{\infty}=(\infty, \infty,0)$. It provides the full collection of rectangles in our model. At time $t$ there is only one  doubly infinite rectangle, whose age is increasing at speed one. This rectangle produces rectangles at times $(T_i)_{i\geq 1}$, which follow a Poisson process with parameter~$1$. The length, width and age of the rectangle produced at time $T_i$, for $i\geq 1$,  is $(\infty,T_{i}-T_{i-1}, 0) $, where $T_0=0$. Focusing again on the finite rectangles, we consider $f$ supported by $\R_+^3$, and get
\begin{align*}
M_tf(z_{\infty})&=\sum_{i\geq 1} \E\left(1_{ T_i\leq t} M_{t-T_i}f(\infty,T_{i}-T_{i-1},0)\right)\\
&=\sum_{i\geq 1} \E\left(G(t-T_{i-1})\right),
\end{align*}
by writing (again) $T_i=T_{i-1}+T_{i}-T_{i-1}$ and conditioning by $T_{i-1}$ and defining 
$$G(v)= \E(1_{\Delta\leq v}M_{v-\Delta}f(\infty,\Delta,0))$$
for $v\in \mathbb R$,  with $\Delta$ exponential random variable with parameter one.
We get
\begin{align}
\label{intemps}
M_tf(z_{\infty})&=G(t)+\int_0^tG(t-s)ds. 
\end{align}
We obtain the following counterpart of Proposition \ref{propinftyl}, using the norm
$$\parallel f\parallel =\sup_{\substack{ x,a,y  \in \R_+, \\ a\leq x}} \vert f(x,y,a) \vert e^{-3(x+y)/4}.$$
Its proves that the asymptotic profile is given by   the following measure on $\R_+^3$, 
  \begin{align}
\Pi(f) &=\int_0^{\infty} e^{-\ell}\Pi_\ell(f) d\ell 
=\int_{\R_+^2} L\ell e^{-(L+\ell)}\, \pi_{(L,\ell,0)}({f}/{h})  \, d\ell \, dL.
\end{align}
where we recall that $\Pi_\ell$ is defined in \eqref{pibardef} and $f$ is a measurable non-negative function.
\begin{prop} \label{propinftyinfty} There exists  a constant $C$ such that for any  measurable function  $f: \mathcal X \rightarrow \R_+$ with support in $\R_+^3$  and such that $\parallel f \parallel<\infty$ and for any $t\geq 0$,
  $$M_tf(\infty,\infty, 0) \leq C\parallel f \parallel (1+t)^2$$  and 
  $$ \big\vert M_tf(\infty,\infty, 0) -t^2\Pi(f) \big\vert\leq  C\parallel f \parallel t.$$
Moreover,  the measure  $\Pi$ on $\R_+^3$  satisfies
$$\Pi(dL,d\ell,da)=g(L,\ell)dLd\ell \delta_L(da)$$
where its density 
 $g$ with respect to Lebesgue measure for the first and second coordinate writes 
\begin{align*}
 g(L,\ell)=\sum_{k\geq 1} \frac{(k+1)}{((k-1)!)^2} \left( ke^{-(k+1)L-k\ell} +\frac{1+k}{k}e^{-(1+k)L-(1+k)\ell}\right).
\end{align*}
\end{prop}
We observe that the limiting profile $\Pi$ here is restricted to frozen rectangle, i.e. rectangles $(L,\ell,a)$ with $a=L$. The population of inactivated rectangle is of order of magnitude $t^2$, while The active rectangles are negligible. Actually, the second estimate of the proposition ensures that their  and the second estimate of their order of magnitude is $t$.
\begin{proof} Using the first estimate of
Proposition \ref{propinftyl},
\begin{align*}
    G(v)&\leq (1+v)\int_0^v \parallel f \parallel_{\ell,\infty}e^{-\ell}(1+\ell)^3 d\ell\\
    &\qquad \qquad\leq 2 (1+v)  \parallel f \parallel   \,\int_0^{\infty} e^{-\ell/4}(1+\ell)^3d\ell , 
    \end{align*}
which proves the first  by recalling \eqref{intemps}. We obtain similarly the second part using now the second estimate of Proposition \ref{propinftyl}.

Let us now identify the limiting measure $\Pi$.
For short, we write $$\pi_{L,\ell}(A)=\pi_{(L,\ell,0)}(A\times \R_+)$$   the measure on $\R_+^2$ giving the limiting  distribution of the length and width of the spine rectangle starting from  $(L,\ell,0)$ with age $0$ defined in \eqref{defpiz}. Using the first time when it branches, this family of (probability) measures satisfies:
$$\pi_{L,\ell}=\int_0^L dae^{-a}\left(\frac{L-a}{L}\pi_{L-a,\ell} +\frac{a}{L}\pi_{\ell,a}\right) +e^{-L}\delta_{L,\ell}$$
for $L,\ell \in \R_+^2$.
Then, its $h$ transform, i.e. the measure  defined by
$$\widehat{\pi}_{L,\ell}(\cdot)=L\ell
\pi_{L,\ell}(\cdot/h),$$ 
satisfies
$$\widehat{\pi}_{L,\ell}=\int_0^L da e^{-a}\left(\widehat{\pi}_{L-a,\ell} +\widehat{\pi}_{\ell,a}\right)+e^{-L}\delta_{L,\ell}.$$
We introduce now the following family of measures on $\R_+^2$, 
$$\Pi_{n,m}=\int_{[0,\infty)^2}e^{-nL-m\ell}\widehat{\pi}_{L,\ell} \, dL d\ell$$
for $n,m\geq 1$ and we want to compute
$\Pi_{1,1}=\Pi$.
The equation on $\widehat{\pi}_{L,\ell}$
yields : 
$$\Pi_{n,m}=\frac{1}{n+1}\Pi_{n,m}+\frac{1}{n}\Pi_{m,n+1}+\int_{[0,\infty)^2}e^{-(n+1)L-m\ell} \, \delta_{L,\ell} dLd\ell$$
We can gather the two terms involving the measure $\Pi_{n,m}$ and iterate
\begin{align*}
\Pi_{n,m}&=\frac{n+1}{n^2}\Pi_{m,n+1}+\frac{n+1}{n}\int_{[0,\infty)^2}e^{-(n+1)L-m\ell}\delta_{L,\ell} dLd\ell\\
&=\frac{(n+1)(m+1)}{n^2m^2}
\Pi_{n+1,m+1}+R_{n,m},
\end{align*}
with $R_{n,m}$ measure on $\R_+^2$ defined by
\begin{align*}
R_{n,m}&=\int_{[0,\infty)^2}\left(\frac{n+1}{n} e^{-(n+1)L-m\ell}+\frac{(n+1)(m+1)}{n^2m} e^{-(m+1)L-(n+1)\ell}\right)\delta_{L,\ell} dLd\ell.
\end{align*}
The series converges since we have an a priori bound on $\Pi_{n,m}$. Indeed by definition,  $\Pi_{n,m}(\R_+^2)$ is dominated by $1/nm$.
We get
$$\Pi_{n,m}=R_{n,m}+\sum_{j=1}^{\infty}\frac{(n+j)(m+j)}{n^2m^2\Pi_{k=1}^{j-1} (n+k)(m+k)} \, R_{n+j,m+j} $$
and $\Pi_{n,m}$ has a density $g_{n,m}$ with respect to Lebesgue measure, that can be computed. More precisely, 
$\Pi_{n,m}= g_{n,m}(L,\ell)dLd\ell,$
where
\begin{align*}
    g_{n,m}(L,\ell)&=\frac{n+1}{n}e^{-(n+1)L-m\ell}\\
    &\quad +\sum_{j=1}^{\infty}\frac{(n+j+1)(m+j)}{n^2m^2\Pi_{k=1}^{j-1} (n+k)(m+k)} e^{-(n+j+1)L-(m+j)\ell}\\
     &\quad + \frac{(n+1)(m+1)}{n^2m} e^{-(m+1)L-(n+1)\ell}\\
 &\quad +\sum_{j=1}^{\infty}\frac{(n+j+1)(m+j+1)}{n^2m^2(n+j)\Pi_{k=1}^{j-1} (n+k)(m+k)} e^{-(n+j+1)L-(m+j+1)\ell}
\end{align*}    
  Letting $n=m=1$ yields the density of $\Pi$ and ends the proof.
\end{proof}

\begin{figure}
\centering
\subcaptionbox{}
{\includegraphics[width=0.49\textwidth]{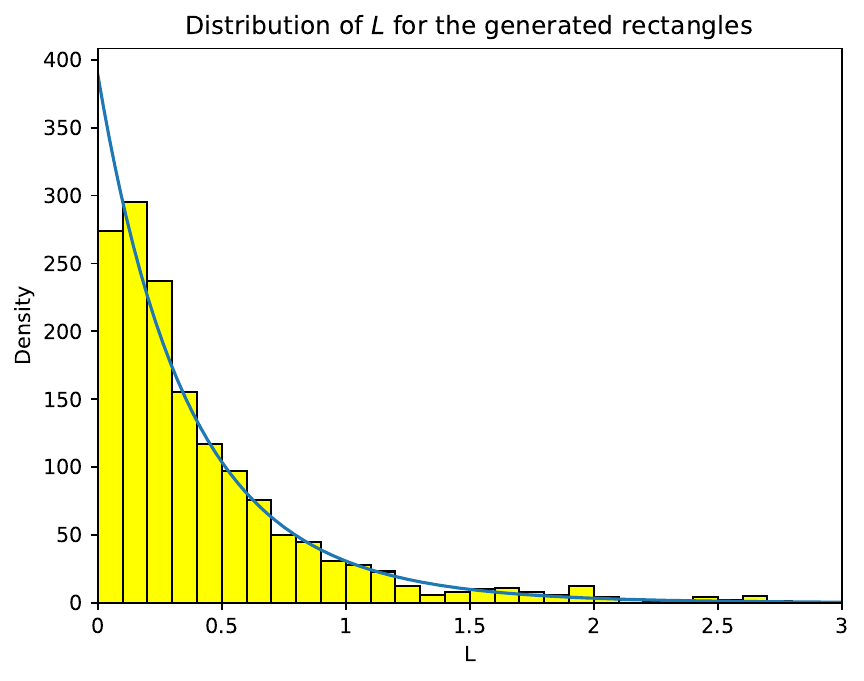}}
\subcaptionbox{}
{\includegraphics[width=0.49\textwidth]{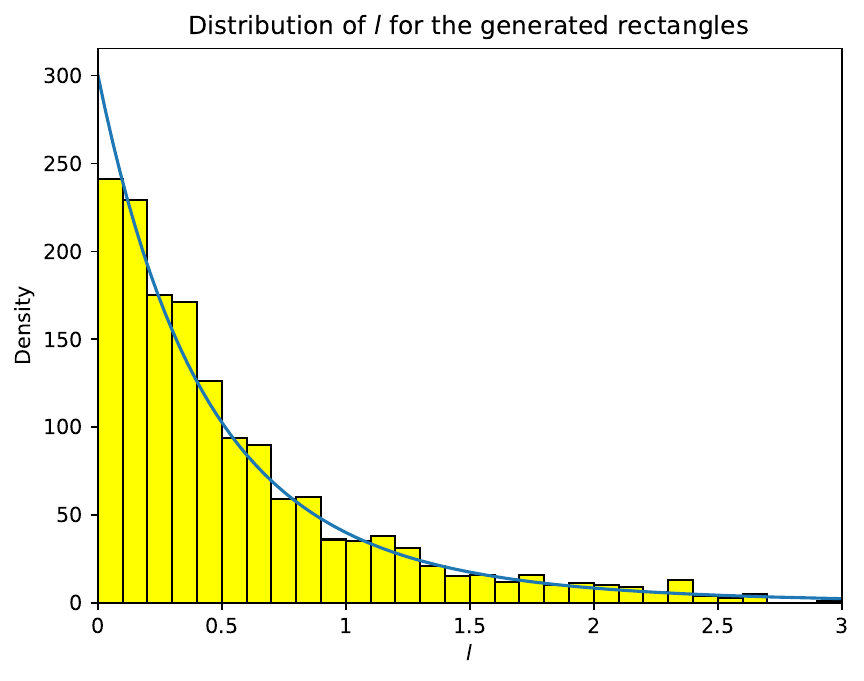}}
\caption{For the network simulation represented in Figure~\ref{fig:simu1}, we represent: (a) distribution of the length $L$ of the rectangles, compared to the formula given by Proposition \ref{propinftyinfty}; (b) distribution of the width $l$ of the rectangles, compared to the formula given by Proposition~ref{propinftyinfty}. }\label{fig:simu2}
\end{figure}

We complement this section by estimating the rectangles with one infinite size created by the original doubly infinite rectangle. We recall that $(T_i)_{i\geq 1}$ is the Poisson point process with parameter $1$ giving the branching times of $z_{\infty}=(\infty,\infty,0)$. We  write  $A_n(t)$ the age at time $t$ of the infinite rectangle,  with width $T_{n+1}-T_n$, whose branch has been created at time $T_n$.  We have for any $\ell_0,a_0 \in \R$, by conditioning on $\mathcal F_{T_{n+1}}$,
\begin{align}
    M_t 1_{\{\infty\}\times  [\ell_0, \infty)\times  [ a_0,\infty)} (z_{\infty})&\leq \sum_{n\geq 0} \E(1_{T_n \leq t, T_{n+1}-T_n\geq \ell_0, A_n(t)\geq a_0}) \nonumber\\
    &\leq \sum_{n\geq 0} \E(1_{T_n \leq t}).\mathbb P(T_{n+1}-T_n\geq \ell_0). \mathbb P(A_n(t)\geq a_0)\nonumber\\
    &\leq e^{-\ell_0-a_0}\E(\#\{ n\geq 1 : T_n\leq t\})\leq t e^{-\ell_0-a_0}. \label{cagrowlin}
\end{align}

\section{$L^2$ computations and convergence}
\label{laviedansL2}
We are now in position to describe the asymptotic behavior of the empirical measure using the branching structure with double level of immigration. We consider the empirical measure restricted to  finite rectangles, whose number grows as $t^2$ as time $t$ goes to infinity. Indeed,   the family of infinite rectangles is easy to describe since it is directly given by Poisson processes. Besides their number is negligible, growing (only) linearly along time. For the statement and quantitative estimates, we need to consider the suitable space functions, inherited from the previous section. Thus, we introduce the following set of functions whose exponential growth is limited :
$$\mathfrak F:=\{ f : \mathcal X \rightarrow \R, \, \, f \, \text{measurable}, \, \text{Support} f\subset \R_+^3, \,   \vertiii f<\infty \}, 
$$
where the norm $\vertiii f$ is defined by
$$\vertiii f:= \sup_{\substack{(L,\ell,a,a')\in \R_+^4\\ a\leq L, a'\leq \ell} }\left\{ \left(\vert f(L,\ell,a)\vert + \vert f(\ell,L,a') \vert \right)  {e^{-(L+\ell)/5}}\right\}.$$
The main result on the asymptotic profile of the empirical measure can be stated as follows
\begin{theorem} \label{Thprinc} For any $f\in \mathfrak F$,
$$\frac{\langle Z_t, f \rangle}{t^2} \stackrel{t\rightarrow\infty}{\longrightarrow}\frac{\Pi(f)}{2 \Pi(h)}$$
in $L^2$.
\end{theorem}
In the proof below, we also evaluate the speed of convergence, which is of order $1/t$. 
We start by proving $L^2$  estimates for "centered" test functions.
\begin{lemme} For any
    $f\in \mathfrak F$ such that $\Pi(f)=0$,
$$\E_{\delta_{(\infty,\infty, 0)}}\left(\left( \frac{\langle Z_t, f \rangle}{t^2}\right)^2\right)=\mathcal O\left(\frac{1}{t}\right).$$
\end{lemme} 

\begin{proof} We start from the initial doubly infinite rectangle $z_{\infty}=(\infty,\infty,0)$
and expand second moment  using branching property. The branching 
event  involves the most recent common ancestor $u\wedge v$ of two rectangles $u,v\in \mathcal V(t)= V(t) \cup W(t)$ alive (active or frozen) at time $t$. The number of alive rectangle is of order $t^2$ and the couples of order $t^4$.
As we will see below and as one may guess from the spatial geometry approximated by a triangle, the  most recent common ancestor  of most couples $(u,v)$ at time $t$ is a fragment   $u\wedge v$ doubly infinite. It means   that most of  $u\ \wedge v\in \mathbb N$  are rectangle of the form $(\infty,\infty,a)$ at branching. 
Most recent common ancestor which have one single infinite size, i.e. rectangles $u\wedge v\in \mathbb N^2$ of the form $(\infty, l,a)$, will be negligible and are expected to be of order of magnitude $t^3$, see below the computations. The most recent common ancestor of other couples $u\wedge v$ are finite rectangles.
The number of such couples will also be negligible is also expected to be  of order $t^2$. Skipping the initial value $\delta_{z_\infty}$ in Notation, we get
\begin{align*}
&\E(\langle Z_t, f \rangle)^2)\\
&\quad = \E\left(\sum_{u\in \mathcal V(t)} f(Z_u(t))^2\right) + \E\left(\sum_{\substack{u,v \in \mathcal V(t), u\neq v,\\
u\wedge v\not\in \mathbb N \cup \mathbb N^2}}f(Z_u(t)) f(Z_v(t))\right)\\
& \qquad+ \E\left(\sum_{\substack{u,v \in \mathcal V(t), u\neq v,\\
 u\wedge v\in \mathbb N^2 }}f(Z_u(t)) f(Z_v(t))\right)  + \E\left(\sum_{\substack{u,v \in \mathcal V(t), u\neq v, \\ u\wedge v\in \mathbb N}}f(Z_u(t)) f(Z_v(t))\right).
\end{align*}
The first term grows at most as $t^2$ by Proposition \ref{propinftyinfty}:
$$\E\left(\sum_{u\in \mathcal V(t)} f(Z_u(t))^2\right)=M_tf^2(z_\infty) \leq C\parallel f^2\parallel (1+t)^2.$$
The second term corresponds to couples of rectangles coming from one finite rectangle. More precisely, we write  $M_s(z_{\infty},dz)$  the kernel measure  associated with the finite moment semigroup  and defined by  $M_s(z_{\infty},A)=M_s1_A(z_\infty)$, for $A$ Borel set of $\R_+^3$. Since only the active rectangle branch,   we get
\begin{align*}
\E\left(\sum_{\substack{u,v \in \mathcal V(t), u\neq v,\\
u\wedge v\not\in \mathbb N \cup \mathbb N^2}}f(Z_u(t)) f(Z_v(t))\right) \leq \int_0^t ds \int_{E} M_s(z_\infty, dz)  M_{t-s}f(\tau^\rightarrow(z))M_{t-s}f (\tau^\curvearrowright(z)),
\end{align*}
where $E=\{(L,\ell,a) \in \R_+^3 : a \leq L\}$ and we recall that $\tau^\rightarrow(L,\ell,a)= (L-a, \ell, 0) \text{ and } \tau^\curvearrowright(L,\ell ,a)= (\ell, a, 0)$. Then using Propositions \ref{job2} $i)$,
$$\max(M_{t-s}f(\tau^\rightarrow(z)), M_{t-s}f(\tau^\curvearrowright(z))) \, \leq \, C(1+L+\ell)^3  e^{(L+\ell)/5}\vertiii f$$
for $f \in  \mathfrak F$.
Writing $\kappa(L,a,\ell)=(1+L+\ell)^6 e^{2(L+\ell)/5}$, $\parallel \kappa\parallel<\infty$  and using that the first part of Proposition \ref{propinftyinfty},  we get $M_s(z_{\infty}, \kappa)\leq  C \parallel \kappa\parallel (1+s)^2$
and
\begin{align*}
\E\left(\sum_{\substack{u,v \in \mathcal V(t), u\neq v,\\
u\wedge v\not\in \mathbb N \cup \mathbb N^2}}f(Z_u(t)) f(Z_v(t))\right)&  \leq C^2\vertiii f ^2\int_0^t ds M_s(z_{\infty}, \kappa)\\
& \leq C^3 \vertiii f^2 \parallel \kappa\parallel (1+t)^3.
\end{align*}
The third term corresponds
to  couples of rectangles coming from a rectangle with one single infinite side. We obtain similarly
\begin{align*}
&\E\left(\sum_{\substack{u,v \in \mathcal V(t), u\neq v,\\
 u\wedge v\in \mathbb N^2 }}f(Z_u(t)) f(Z_v(t))\right) \\
 &\qquad =\int_0^t ds \int_{\R^2} M_s(z_{\infty}, \{\infty \} d\ell da)  M_{t-s}f(\infty,\ell,0)M_{t-s}f (\ell,a,0).
\end{align*}
We use again Proposition \ref{job2} $i)$ for $M_{t-s}f (\ell,a,0)$ and the first part of Proposition \ref{propinftyl} for  $M_{t-s}f(\infty,\ell,0)$. Thus, writing  $\zeta(L,\ell,a)= 1_{L=\infty, \ell<\infty}(1+\ell+a)^3\exp((\ell+a)/5)(1+\ell)^3\exp(3 \ell/4),$ we get 
$$1_{L=\infty, \ell<\infty}M_{t-s}f(\infty,\ell,0)M_{t-s}f (\ell,a,0)\leq C^2\vertiii  f \parallel f \parallel \zeta(L,\ell,a)(1+t-s).$$
Adding that  $M_s(z_{\infty},\zeta)$ grows (at most) linearly with $s$ from $\eqref{cagrowlin}$, we obtain 
\begin{align*}
\E\left(\sum_{\substack{u,v \in \mathcal V(t), u\neq v,\\
 u\wedge v\in \mathbb N^2 }}f(Z_u(t)) f(Z_v(t))\right)
 & \leq C^2 \vertiii  f \parallel f \parallel \int_0^t M_s(z_{\infty},\zeta)(1+t-s) ds\\
 &\leq C' \vertiii  f \parallel f \parallel (1+t)^3.
 \end{align*}
Let us turn to the last term, which is providing the main contribution.
\begin{align*}
&\E\left(\sum_{\substack{u,v \in \mathcal V(t), u\neq v, \\ u\wedge v\in \mathbb N}}f(Z_u(t)) f(Z_v(t))\right)\\
&\qquad =\sum_{i\geq 1}\E\left(1_{T_i\leq t}M_{t-T_i}f(\infty,\infty,0)M_{t-T_i}f(\infty,T_i-T_{i-1},0)\right)\\
&\qquad =\sum_{i\geq 1}\E\left(F(t-T_{i-1})\right)=F(t)+\int_0^t F(t-s)ds ,
\end{align*}
where, writing $\Delta=_{law} T_{i}-T_{i-1}$ an exponential r.v. with parameter $1$
$$F(u)=\E(1_{\Delta\leq u} M_{u-\Delta}f(\infty,\infty,0) M_{u-\Delta}f(\infty,\Delta,0)).$$
For $f$ non-negative, $F(u)$ is expected to be of order $u^2.u=u^3$ and integrating in time the last term is of order $t^4$. To prove $L^2$ convergence to the stationary distribution given by $\Pi$, we consider now $f$ such that 
$$\Pi(f)=0$$
and obtain a lower order of magnitude.
Using the second part of Proposition \ref{propinftyinfty}
for $M_{u-\Delta}f(\infty,\infty,0)$
and the first part of Proposition \ref{propinftyl} for $M_{u-\Delta}f(\infty,\Delta,0)$, we   get
\begin{align*}
 \vert F(u)\vert &\leq \parallel f\parallel \E(1_{\Delta\leq u} (1+u-\Delta)^2 (1+\Delta)^3 \parallel f\parallel_{\Delta,\infty}) \\ 
 & \leq  \parallel f\parallel^2 (1+u)^2 \E( (1+\Delta)^3e^{3\Delta/4}) .
\end{align*}
Adding that $\E( (1+\Delta)^3e^{3\Delta/4})<\infty$  and integrating $(1+u)^2=(1+t-s)^2$ for $s\in [0,t]$, we get  
\begin{align*}
&\E_{z_0}\left(\sum_{\substack{u,v \in \mathcal V(t), u\neq v, \\ u\wedge v\in \mathbb N}}f(Z_u(t)) f(Z_v(t))\right)\leq  C \parallel f\parallel^2(1+t)^3.
\end{align*}
Gathering these results 
proves the result.
\end{proof}
We consider now $g\in \mathfrak F$ 
and $f=g-\Pi(g)/\Pi(1)$, then $f\in \mathfrak F$, $\vertiii f$ and  $\parallel f\parallel$ are finite and  $\Pi(f)=0$.
Applying the previous lemma to $f$,   we obtain
$$\E_{z_0}\left(\left( \frac{\langle Z_t, g \rangle- \langle Z_t, 1 \rangle \Pi(g)/\Pi(1)}{t^2}\right)^2\right)=\mathcal O\left(\frac{1}{t}\right).$$
Thus,
\begin{equation}
\label{limgl2}
\frac{\langle Z_t, g \rangle}{t^2}-\frac{\langle Z_t, 1 \rangle}{t^2}\frac{\Pi(g)}{\Pi(1)} \stackrel{t\rightarrow \infty}{\longrightarrow }   0 \qquad \text{in } L^2 
\end{equation}
Thus holds in particular for $g=h\in \mathfrak F$ the surface function, which will allow us to identify the convergence of $\langle Z_t, 1 \rangle/t^2$ through the convergence of
$\langle Z_t, h \rangle/t^2$ and conclude. 
Let us thus consider the total surface of finite rectangles,  which is easy to describe thanks to the conservation of surface at branching events. 
\begin{lemme}
\label{idtmart}
Starting from $Z_0=\delta_{(\infty,\infty, 0)}$,    the following limit holds
   $$\frac{\langle Z_t, h \rangle}{t^2} =\frac{1}{t^2}\sum_{u\in V(t) \cup W(t) } L_u\ell_u1_{L_u\ell_u<\infty}\stackrel{t\rightarrow\infty}{\longrightarrow} \frac{1}{2}$$
in $L^2$ and the difference goes to $0$ in $L^2$ at least at speeed $1/t$ as $t$ goes to infinity.
\end{lemme}
\begin{proof} The original fragment grows at speed one from the origin and the surface 
$\langle Z_t, h \rangle$ in included in the rectangle triangle  with size $t$, whose surface is $t^2/2$. We can even control the missing part, it corresponds to the infinite rectangles up to time $t$. The missing part at time $t$ is thus dominated by independent rectangles whose size is exponentially distributed  and
$$0\leq t^2/2-\langle Z_t, h \rangle \leq \Delta_t \quad \text{p.s.}, \qquad \Delta_t\stackrel{law}{=} \sum_{i=1}^{N_t} R_i$$
where $N$ is a Poisson Process with parameter $1$ and $R_i$ are i.i.d. r.v. distributed as the product of two exponential independent r.v. Adding that 
$\sum_{i=1}^{N_t} R_i/t^2$ goes to $0$ in $L^2$ as $1/t$, by  a classical direct computation ($L^2$ law of large numbers),  ends the proof.
\end{proof}
Putting the pieces, we can   prove  Theorem
\ref{Thprinc}.
\begin{proof}[Proof of Theorem \ref{Thprinc}]
Applying \eqref{limgl2} to $g=h$ shows that
$$\lim_{t\rightarrow\infty} \frac{\langle Z_t, h \rangle}{t^2}-\frac{\langle Z_t, 1 \rangle}{t^2}\frac{\Pi(h)}{\Pi(1)}=0=\frac{1}{2}- \lim_{t\rightarrow\infty} \frac{\langle Z_t, 1 \rangle}{t^2}\frac{\Pi(h)}{\Pi(1)}$$
in $L^2$ by using Lemma \ref{idtmart}. Then 
$ \eqref{limgl2}$ becomes
$$\frac{\langle Z_t, g \rangle}{t^2}-\frac{\Pi(1)}{2\Pi(h)}\frac{\Pi(g)}{\Pi(1)} \stackrel{t\rightarrow \infty}{\longrightarrow }   0 \qquad \text{in } L^2, $$
which ends the proof.
\end{proof}

\section{Discussion}
\label{sec:discussion}
Let us discuss some perspectives and stimulating questions. First, on the mathematical side of the study of the model, we expect to be able to describe more finely the genealogy and the long time behavior. In particular, it would be interesting to study precisely the genealogy of inactive rectangles and the genealogy of active interior rectangles and the genealogy at the boundary. Besides, while our  $L^2$ estimates are expected to be sharp and allow to quantify the speed of convergence, the a.s. behavior is left open. The speed of convergence $1/t$  seems to prevent a direct proof using Borel Cantelli type argument, but the freeze of rectangles and martingales should help, at least for the a.s. behavior of the full empirical measure.\\

The model we consider in this manuscript is a simplification of the model introduced in \cite{Hannezo} and described above: we assume that the active tips progress on straight lines, that branching at the tip has a particular structure: of the two active tips that appear, one continues on the initial straight line, while the other one follows the rotation of this line at an angle $\pi/2$. This particular rule allowed us to show that the 2D surface $[0,\infty)^2$ undergoes a particular fracturation process. More precisely, we were able to exhibit a branching process on rectangles with an age describing this fracturation, and to analyse the long time dynamics of that process. A possible generalization of the model is to assume that during branching events, one tip continues along the straight line (as it is the case at the moment, and that the other tip emerges on the right-hand-side, but at an angle that is not necessarily $\frac \pi 2$: it could for instance follow a law supported in $(0,\pi)$. We believe it could be possible to describe the extend the approach we introduced into a branching process on polygons with age. The most interesting generalization would however to allow branching on both sides of a vessel. It is unfortunately unclear that it would then be possible to describe the fracturation through independent units of surface, and the connection with a branching process would then break down. New ideas will be necessary to investigate such processes, but we hope some of the properties we exhibited in this manuscript will hold in a more general context (for instance the specific role played by some vessels that are the first to arrive in specific parts of the plane). 

\section{Appendix}

\subsection{The PDE associated to the first moment of the empirical measure}
\label{sec:APP1}
Recalling \eqref{eq:semi-group}, considering bounded measurable function  $f$ which does noes not depend on the last variable $u\in \mathcal U$ and is $C^1$ with respect to the third coordinate (the age $a$),  $f :  (L,\ell,a) \in \mathcal X \rightarrow   f(L,\ell,a) \in \mathbb R_+$, 
\begin{align*}
&\langle {Z}_t,f \rangle= \langle {Z}_0,f \rangle + \int_0^t  \int_{\mathcal X}  \frac{\partial f}{\partial a} ( L,\ell,a) 1_{a<L} {Z}_s (dL,d\ell,da) \,ds + M_t^f   \label{eq:semi-group}\\
&\, \qquad \qquad  +\int_0^t  \int_{\mathcal X} 1_{a<L} \Big(f(\tau^\rightarrow(L,\ell,a))
+f(\tau^\curvearrowright(L,\ell,a))  -f(L,\ell,a)\Big) {Z}_s (dL,d\ell,da) \, ds.\nonumber 
\end{align*}
Considering the measure $\gamma_t$ on $\mathcal X$ defined by
$$\E_{\delta_{(\infty,\infty,0)}}(\langle  Z_t,f \rangle) =\int_{\mathcal X} f(L,\ell,a)\,\gamma_{t}(dL,d\ell,da)=M_tf(\infty,\infty,0)$$
and using that $M^f$ is a martingale starting from $0$, 
we obtain by taking expectation that
\begin{align*}
&\gamma_t(f)= f(\infty,\infty,0)  + \int_0^t  \int_{\mathcal X}  \frac{\partial f}{\partial a} ( L,\ell,a) 1_{a<L} \gamma_s (dL,d\ell,da) \,ds \\
&\, \qquad \qquad  +\int_0^t  \int_{\mathcal X} 1_{a<L} \Big(f(\tau^\rightarrow(L,\ell,a))
+f(\tau^\curvearrowright(L,\ell,a))  -f(L,\ell,a)\Big) \gamma_s (dL,d\ell,da) \, ds.\nonumber 
\end{align*}
Moreover $\gamma_{t}$ admits the following form
\begin{align*}
\gamma_{t}(dL,d\ell,da)&=n_{t,L,\ell}(da)1_{(L,\ell)\in\mathbb R_+^2}1_{a<L}\,dL\,d\ell+m_{t,\ell}(da)1_{\ell\in \mathbb R_+}\, \delta_{\infty}(dL)d\ell\,\\
&\quad  +p_t(da)\delta_{(\infty,\infty)}(dL,d\ell).
\end{align*}
Let us then  consider $f^n\in C^1_c(\mathbb R_+^3)$, $f^m\in C^1_c(\mathbb R_+^2)$, $f^p\in C^1_c(\mathbb R_+)$, and introduce the following test functions \[f(L,\ell,a):=f^n(L,\ell,a)1_{(L,\ell)\in\mathbb R_+^2}+f^m(\ell,a)1_{L=\infty,\ell\in\mathbb R_+}+f^p(a)1_{(L,\ell)=(\infty,\infty)}.\]

Then, thanks to a changes of variable, we obtain:
\begin{equation}
\left\{\begin{array}{l}
\int_{\mathbb R_+^3} f^n(L,\ell,a)\,dn_{t,L,\ell}(a)\,dL\,d\ell =
 \int_0^t \int_{\mathbb R_+^3}  \frac{\partial f^n}{\partial a} (L,\ell,a)1_{a<L}\,n_{s,L,\ell}(da)\,dL\,d\ell\,ds\nonumber\\
\qquad +\int_0^t\Big[\int_{\mathbb R_+^2}f^n(L,\ell,0) \int_{\mathbb R_+}\,n_{s,L+\tilde a,\ell}( d\tilde a)\,dL\,d\ell+\int_{\mathbb R_+^2} \int_{0}^{\tilde L} f^n(L,\ell,0)\,n_{s,\tilde L,L}(d\ell)\,d\tilde L\,dL\\
\qquad\phantom{+\int_0^t\Big[} -\int_{\mathbb R_+^3} f^n(L,\ell,a)1_{a<L}\,dn_{s,L,\ell}(a)\,dL\,d\ell+\int_{\mathbb R_+^2}f^n(L,\ell,0) \,dm_{s,L}(\ell)\,dL\Big]\,ds,\\ \\
\int_{\mathbb R_+^2}  f^m(\ell,a)\,m_{t,\ell}(da)\,d\ell=
\int_0^t \int_{\mathbb R_+^2} \frac{\partial f^m}{\partial a} (\ell,a)\,dm_{s,\ell}(a) \,d\ell\, ds\nonumber\\
\qquad +\int_0^t \Big[\int_{\mathbb R_+^2} f^m(\ell,0) \,m_{s,\ell}(d\tilde a)\,d\ell-\int_{\mathbb R_+^2} f^m(\ell,a)\,m_{s,\ell}(da)\,d\ell+\int_{\mathbb R_+} f^m(\ell,0)\,d p_s(\ell)\Big]\,ds\\ \\
\int_{\mathbb R_+}  f^p(a) \,dp_t(a)=f^p(0)
+ \int_0^t \int_{\mathbb R_+}  \frac{\partial f^p}{\partial a} (a)\,dp_s(a)\, ds\nonumber\\ 
\qquad +\int_0^t \Big[-\int_{\mathbb R_+} f^p(a)\,p_s(da)+f^p(0)\int_{\mathbb R_+} \,p_s(\tilde da)\Big]\,ds,
\end{array}\right.
\end{equation}

If $\phi\in C^2_c(\mathbb R_+^2)$, then we can use $\partial_t \phi$ as a test function for the third equation in the system above,
to show
\begin{align*}
    &-\int_0^\infty \int_{\mathbb R}(\partial_t \phi)(t,a)\,dp_t(a)\,dt=\int_{\mathbb R}\phi(0,a)\,dp_0(a)+\int_{\mathbb R_+\times\mathbb R}  \partial_a\phi(s,a)1_{a<L}\,dp_s(a)\,ds\\
    &\qquad +\int_0^\infty\Big[-\int_{\mathbb R} \phi(s,a)\,dp_s(a)+ \phi(s,0)\int_{\mathbb R} \,dp_s(\tilde a)\Big]\,ds,
\end{align*}
and $p_t$ is therefore a solution in the sense of Definition~4.1 in \cite{santambrogio2015optimal} of the following continuity equation with a measure source term:
\begin{align*}\label{eq:p}
    \left\{\begin{array}{l}
    \partial_t p_t(a)+\partial_a p_t(a)=-p_t(a)+\left(\int_{\mathbb R_+} \,dp_s(\tilde a)\right)\delta_{a=0},\quad (t,a)\in (0,\infty)\times\mathbb R,\\
    p_0(a)=\delta_{a=0},\quad a\in\mathbb R.
    \end{array}\right.
\end{align*}
Note that this equation for $(t,a)\in\mathbb R_+\times\mathbb R$ with a Dirac source term at $a=0$ is equivalent to the transport equation for $p_t$ stated in \eqref{eq:nmp} (which we do not detail here), where $(t,a)\in\mathbb R_+\times\mathbb R_+$ and the source term is replaced by a boundary condition at $a=0$. Then $p_t$ satisfies 
\begin{equation*}
\left\{\begin{array}{l}\frac{\partial p_t}{\partial t}(a)+\frac{\partial p_t}{\partial a}(a)=-p_t(a),\quad a\in(0,\infty), \vspace{0.25cm}\\
p_t(0)=\int_0^\infty \,dp_t(a).
\end{array}\right.
\end{equation*}
Similar arguments can be used for $m_{t,\ell}$ (with $\ell>0$ as a parameter) and $n_{t,L,\ell}$ (with $L,\ell>0$ as parameters). Then, it is possible to see that the satisfy the following system in the sense of Definition~4.1 in \cite{santambrogio2015optimal}:
\begin{equation}\label{eq:nmp}
\left\{\begin{array}{l}
\frac{\partial n_{t,L,\ell}}{\partial t}(a)+1_{a<L}\frac{\partial n_{t,L,\ell}}{\partial a}(a)=-n_{t,L,\ell}(a),\quad (t,a)\in(0,\infty)\times (0,L], \vspace{0.25cm}\\
n_{t,L,\ell}(0)=\int_0^{\infty} \,dn_{t,L,\ell}(a)+\int_{\ell}^{\infty}n_{t,x,\ell}(\ell)\,dx+m_{t,L}(\ell), \quad t\geq 0,\vspace{0.25cm}\\
\frac{\partial m_{t,\ell}}{\partial t}(a)+\frac{\partial m_{t,\ell}}{\partial a}(a)=-m_{t,\ell}(a),\quad (t,a)\in(0,\infty)^2, \vspace{0.25cm}\\
m_{t,\ell}(0)=\int_0^{\infty} \,dm_{t,\ell}(a)+p_t(\ell), \quad t\geq 0, \vspace{0.25cm}\\
\frac{\partial p_t}{\partial t}(a)+\frac{\partial p_t}{\partial a}(a)=-p_t(a),\quad  (t,a)\in(0,\infty)^2\vspace{0.25cm}\\
p_t(0)=\int_0^\infty \,dp_t(a), t\geq 0,
\end{array}\right.
\end{equation}
where these equations hold for any $L,\ell>0$, with these seen as parameters. Moreover, we have the initial condition $n_{0,L,\ell}(da)= 0$, $m_{t,\ell}(da)= 0$ and $p_0(da)=\delta_0(da)$. Notice that $n_{t,L,\ell}(a)$, for $a\in[0,L)$, represents the density of active branches present in the system (corresponding to the set $V(t)$), while $n_{t,L,\ell}(L)$ represents the density of frozen branches present in the system (corresponding to the set $W(t)$).

\medskip

Finally, we can obtain explicit formula for $p_t$ and $m_{t,\ell}$. To show this, we use the weak formulation of the equation satisfied by $p_t$ and a test function equal to $f^p\equiv 1$ on the support of $p_t$ to show 
\begin{align*}
    \int_{\mathbb R_+}  \,dp_t(a)=1
+ \int_0^t \Big[-\int_{\mathbb R_+} \,dp_s(a)+\int_{\mathbb R_+} \,d p_s(\tilde a)\Big]\,ds=1.
\end{align*}
The  characteristic method detailed in \cite{santambrogio2015optimal} can then be used to identify $p_t$: 
\begin{align*}
    p_t&=e^{-t}Y_t{}_{\#} p_0+\int_0^t e^{-(t-s)}Y_{t-s}{}_{\#} \left(\left(\int_{\mathbb R_+} \,dp_s(\tilde a)\right)\delta_{a=0}\right)\,ds
\end{align*}
where $Y_t(a):=a+t$ and $Y_t{}_{\#} p_0$ the push forward of $p_0$ by $Y_t$ (see \cite{villani2008optimal}). Therefore  $p_t(da)=e^{-t}\delta_t(da)+e^{-a}1_{a<t}\,da$.
Next, we can consider the weak formulation satisfied by $m_{t,\ell}$ and  a test function $f^m(\ell,a)f^m(\ell)$ on the support of $m_{t,\ell}$ to show
\[\int_{\mathbb R_+} \,dm_{t,\ell}(a)=\left(1+(t-\ell)\right)1_{\ell\leq t}e^{-\ell}.\]
Here also, the characteristic method can be employed to obtain the following formula:
\[m_{t,\ell}(da)=\left(2+(t-a-\ell)\right)e^{-\ell-a}1_{t-a-\ell\geq 0}\,da+e^{-t}\delta_{t-\ell}(da).\]
For $n_{t,\ell,L}$, we did not succeed to obtain an explicit formula. Notice that for $t,\ell,L>0$, the measure $n_{t,\ell,L}(da)$ over $[0,L]$ can be decomposed into $n_{t,\ell,L}(da)1_{a<L}$, that corresponds to a density of branches alive at time $t$ (see the set $V(t)$ of alive branches defined by \eqref{Vt}), and $n_{t,\ell,L}(da)1_{a=L}$, that corresponds to a density of branches frozen at time $t$, see the set $W(t)$ of alive branches defined by \eqref{Wt}.

\subsection{Numerics}
\label{sec:APP2}

To produce Figure~\ref{fig:simu1} and \ref{fig:simu2} have been produced thanks to a variation of a code written by Ana Fernández Baranda and Nina Gregorio, available on Github, \href{https://github.com/anafdz25/Branching-networks}{https://github.com/anafdz25/Branching-networks}. The code simulate the process described in the introduction, with branching events creating a lateral branch on the right hand side. We can also use these simulations to compute properties of the network created, reproducing the types of graphs that were represented in \cite{Hannezo}, that is an inspiration for this work. We have plotted the corresponding graphs in Figure~\ref{fig:simustatistics}.

\medskip

As described in the discussion, a natural generalization of the model we consider here is to allow branching on both sides. The code above can be modified to include branching on both sides, leading to the simulation represented in Figure~\ref{fig:simuboth}. If branching on both the right and left are allowed, then two active tips cross each other with a positive probability, leading to Figure~\ref{fig:simuboth}(a), were closed rectangles are created (one can show that these events are not possible in the model with branching on one side only, which is what we consider in this manuscript). This surprising feature results from the very simple model that we consider : the fragments grow at a constant speed. These can be seen as artefacts, and we tested numerically two possible modifications of the model: in Figure~\ref{fig:simuboth}(b) we assume that active two active tips coming across do not annihilate each other, but are both able to continue growing: the fragments are then able to cross. In Figure~\ref{fig:simuboth}(c), a different modification is proposed: when two active tips cross, the filament that as produced with the least lateral branching continues growing, while the other one is stopped.

\begin{figure}
\centering
\subcaptionbox{}
{\includegraphics[width=0.49\textwidth]{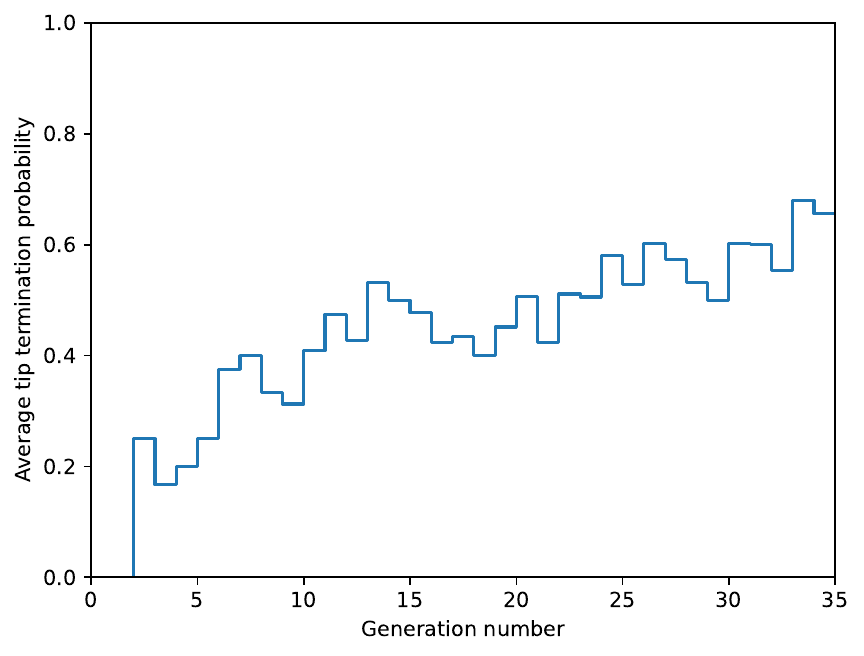}}
\subcaptionbox{}
{\includegraphics[width=0.49\textwidth]{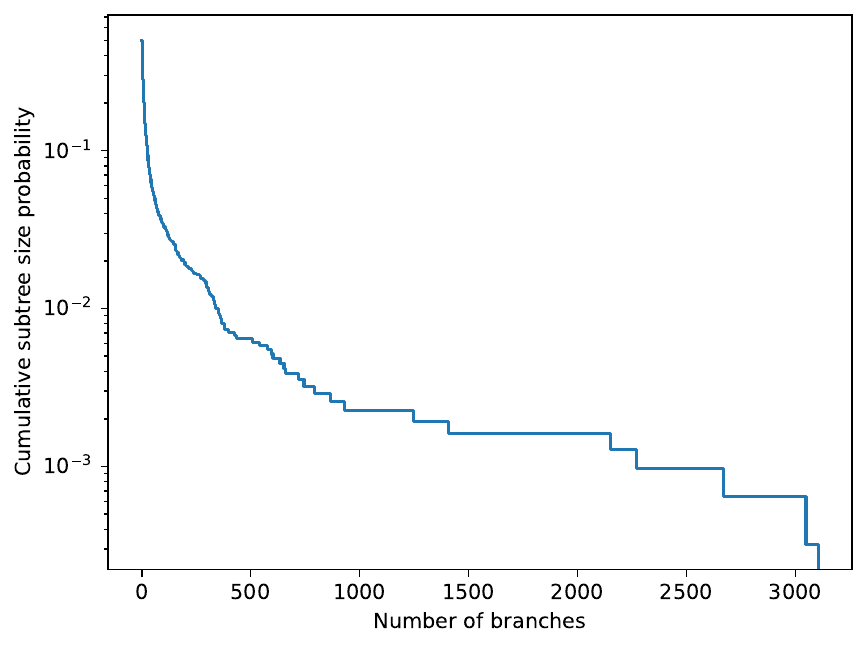}}
\subcaptionbox{}
{\includegraphics[width=0.49\textwidth]{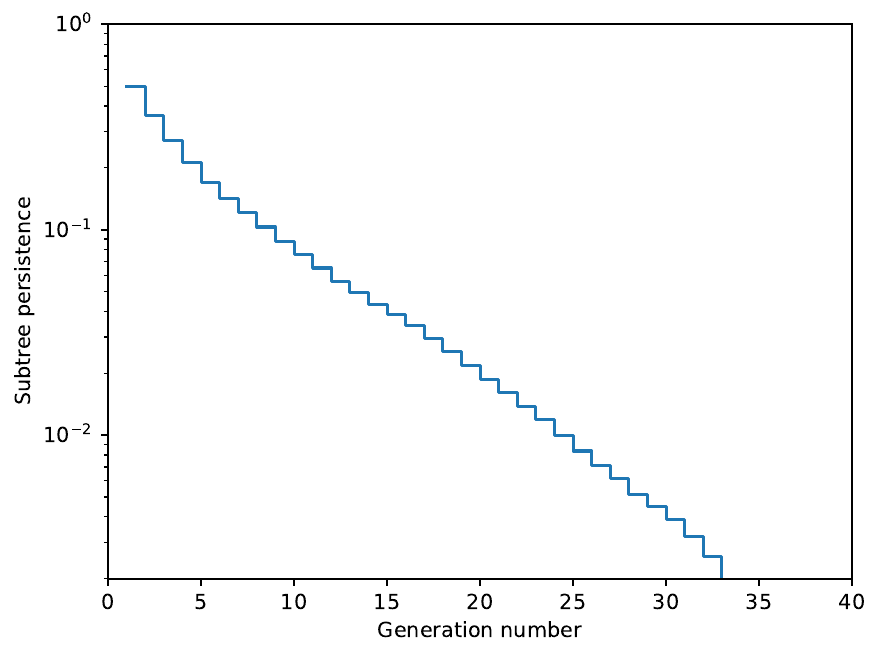}}
\caption{Properties of the genealogical tree represented in Figure~\ref{fig:simu1} (b), that corresponds to the network simulation represented in Figure~\ref{fig:simu1}(a). We have represented quantities considered in \cite{Hannezo}: (a) represents the fraction of nodes at each generation that are leafs; (b) represents, for a node uniformly chosen in the tree, the probability that the subtree it generates contains more than a certain number of nodes (or \emph{branches}); (c) fraction of the nodes at each generation of the tree that are leafs.}\label{fig:simustatistics}
\end{figure}

\begin{figure}
\centering
\subcaptionbox{}
{\includegraphics[width=0.32\textwidth]{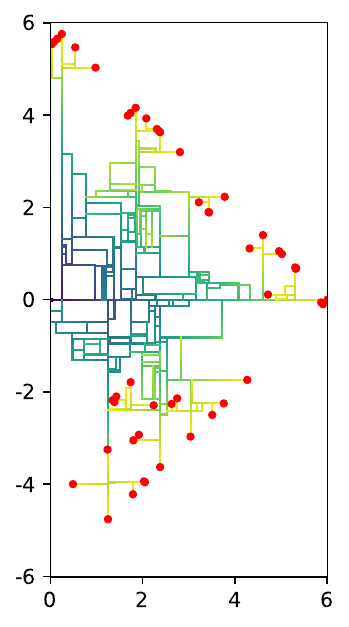}}
\subcaptionbox{}
{\includegraphics[width=0.32\textwidth]{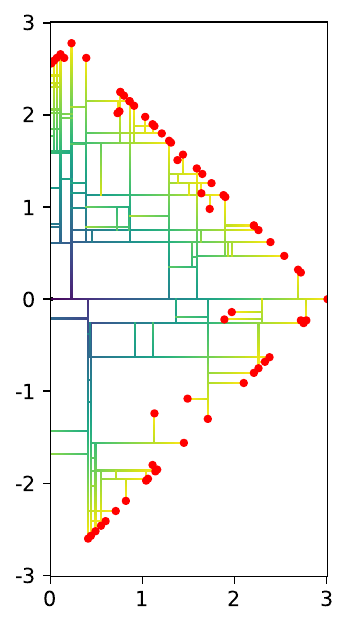}}
\subcaptionbox{}
{\includegraphics[width=0.32\textwidth]{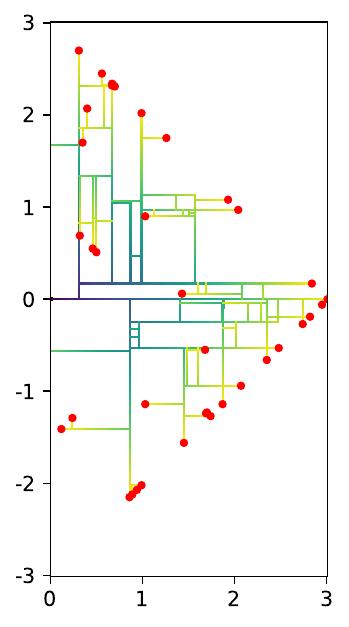}}
\caption{Simulation with branching on both sides, for three possible generalisations of the model considered in this manuscript.}\label{fig:simuboth}
\end{figure}

$\newline$

{\bf Aknowledgment.} The authors would like to warmly thank their beloved colleague, Igor Kortchemski, who moved from our lab  (CMAP, Ecole polytechnique) to ENS Paris, and who previously suggested a nice and efficient way to handle the computations involved in the estimation of small fragments and in the proof of Proposition \ref{prop:oneDimControl}.
This work  was partially funded by the Chair “Mod\'elisation Math\'ematique et Biodiversit\'e" of VEOLIA-Ecole Polytechnique-MNHN-F.X., by the European Union (ERC, SINGER, 101054787), by the Fondation Mathématique Jacques Hadamard and by the project NEMATIC (ANR-21-CE45-0010) and  by the project NOLO (ANR-
20-CE40- 0015), funded by the French Ministry of Research. Views and opinions expressed are however those of the author(s) only and do not necessarily reflect those of the European Union or the European Research Council. Neither the European Union nor the granting authority can be held responsible for them.

\bibliographystyle{plain}
\bibliography{Biblio}

@book{villani2008optimal,
  title={Optimal transport: old and new},
  author={Villani, C{\'e}dric and others},
  volume={338},
  year={2008},
  publisher={Springer}
}

@article{olivier2016does,
  title={How does variability in cells aging and growth rates influence the malthus parameter?},
  author={Olivier, Adelaide},
  journal={arXiv preprint arXiv:1602.06970},
  year={2016}
}

@article{gabriel2018steady,
  title={Steady distribution of the incremental model for bacteria proliferation},
  author={Gabriel, Pierre and Martin, Hugo},
  journal={arXiv preprint arXiv:1803.04950},
  year={2018}
}

@article{doumic2020purely,
  title={A purely mechanical model with asymmetric features for early morphogenesis of rod-shaped bacteria micro-colony},
  author={Doumic, Marie and Hecht, Sophie and Peurichard, Diane},
  journal={arXiv preprint arXiv:2008.04532},
  year={2020}
}

@book{perthame2007transport,
  title={Transport equations in biology},
  author={Perthame, Beno{\^\i}t},
  year={2007},
  publisher={Springer}
}

@article{fournier2006well,
  title={Well-posedness of Smoluchowski's coagulation equation for a class of homogeneous kernels},
  author={Fournier, Nicolas and Lauren{\c{c}}ot, Philippe},
  journal={Journal of functional Analysis},
  volume={233},
  number={2},
  pages={351--379},
  year={2006},
  publisher={Elsevier}
}

@article{santambrogio2015optimal,
  title={Optimal transport for applied mathematicians},
  author={Santambrogio, Filippo},
  year={2015},
  publisher={Springer}
}

@article{zawko,
  title={Crystal templating dendritic pore networks and fibrillar microstructure into hydrogels},
  author={Zawko, Scott A and Schmidt, Christine E},
  journal={	Acta Biomater.},
  volume={6},
  number={7},
  pages={2415--2421},
  year={2010},
  publisher={Elsevier}
}

@article{Katifori,
  title={Damage and fluctuations induce loops in optimal transport networks},
  author={Katifori, Eleni and Sz{\"o}ll{\H{o}}si, Gergely J and Magnasco, Marcelo O},
  journal={Phys. Rev. Lett.},
  volume={104},
  number={4},
  pages={048704},
  year={2010},
  publisher={APS}
}

@article{Scoffoni,
  title={Decline of leaf hydraulic conductance with dehydration: relationship to leaf size and venation architecture},
  author={Scoffoni, Christine and Rawls, Michael and McKown, Athena and Cochard, Herv{\'e} and Sack, Lawren},
  journal={Plant Physiol.},
  volume={156},
  number={2},
  pages={832--843},
  year={2011},
  publisher={American Society of Plant Biologists}
}

@article{Tomasevic,
  title={Ergodic behaviour of a multi-type growth-fragmentation process modelling the mycelial network of a filamentous fungus},
  author={Toma{\v{s}}evi{\'c}, Milica and Bansaye, Vincent and V{\'e}ber, Amandine},
  journal={ESAIM: Probab. Stat.},
  volume={26},
  pages={397--435},
  year={2022},
  publisher={EDP Sciences}
}

@article{Ronellenfitsch,
  title={Global optimization, local adaptation, and the role of growth in distribution networks},
  author={Ronellenfitsch, Henrik and Katifori, Eleni},
  journal={Phys. Rev. Lett.},
  volume={117},
  number={13},
  pages={138301},
  year={2016},
  publisher={APS}
}

@article{Courta,
  title={Mathematics and morphogenesis of cities: A geometrical approach},
  author={Courtat, Thomas and Gloaguen, Catherine and Douady, Stephane},
  journal={	Phys. Rev. E},
  volume={83},
  number={3},
  pages={036106},
  year={2011},
  publisher={APS}
}

@article{Hogan,
  title={Morphogenesis},
  author={Hogan, Brigid LM},
  journal={Cell},
  volume={96},
  number={2},
  pages={225--233},
  year={1999},
  publisher={Elsevier}
}

@article{Matos,
  title={Leaf venation network architecture coordinates functional trade-offs across vein spatial scales: evidence for multiple alternative designs},
  author={Matos, Ilaine Silveira and Boakye, Mickey and Niewiadomski, Izzi and Antonio, Monica and Carlos, Sonoma and Johnson, Breanna Carrillo and Chu, Ashley and Echevarria, Andrea and Fontao, Adrian and Garcia, Lisa and others},
  journal={New Phytol.},
  volume={244},
  number={2},
  pages={407--425},
  year={2024},
  publisher={Wiley Online Library}
}

@article{Bartelemy,
  title={Spatial networks},
  author={Barth{\'e}lemy, Marc},
  journal={Phys. Rep.},
  volume={499},
  number={1-3},
  pages={1--101},
  year={2011},
  publisher={Elsevier}
}

@article{Zubler,
  title={A framework for modeling the growth and development of neurons and networks},
  author={Zubler, Frederic and Douglas, Rodney},
  journal={Front. comput. neurosci.},
  volume={3},
  pages={757},
  year={2009},
  publisher={Frontiers}
}

@article{Kuwata,
  title={A generalised spatial branching process with ancestral branching to model the growth of a filamentous fungus},
  author={Kuwata, Lena},
  journal={arXiv preprint arXiv:2409.13627},
  year={2024}
}

@article{Hannezo,
  title={A unifying theory of branching morphogenesis},
  author={Hannezo, Edouard and Scheele, Colinda LGJ and Moad, Mohammad and Drogo, Nicholas and Heer, Rakesh and Sampogna, Rosemary V and Van Rheenen, Jacco and Simons, Benjamin D},
  journal={Cell},
  volume={171},
  number={1},
  pages={242--255},
  year={2017},
  publisher={Elsevier}
}

@article{Catellier,
author = {Catellier, R\'{e}mi and D’Angelo, Yves and Ricci, Cristiano},
title = {A mean-field approach to self-interacting networks, convergence and regularity},
journal = {Mathematical Models and Methods in Applied Sciences},
volume = {31},
number = {13},
pages = {2597-2641},
year = {2021},
doi = {10.1142/S0218202521500573},

URL = { 
    
        https://doi.org/10.1142/S0218202521500573
    
    

},
eprint = { 
    
        https://doi.org/10.1142/S0218202521500573
    
    

}
}

@article{Kreten2,
  title={Convective stability of the critical waves of an FKPP-type model for self-organized growth},
  author={Kreten, Florian},
  journal={	J. Math. Biol.},
  volume={90},
  number={3},
  pages={33},
  year={2025},
  publisher={Springer}
}

@article{Menshykau,
  title={Image-based modeling of kidney branching morphogenesis reveals GDNF-RET based Turing-type mechanism and pattern-modulating WNT11 feedback},
  author={Menshykau, Denis and Michos, Odyss{\'e} and Lang, Christine and Conrad, Lisa and McMahon, Andrew P and Iber, Dagmar},
  journal={	Nat. Commun.},
  volume={10},
  number={1},
  pages={239},
  year={2019},
  publisher={Nature Publishing Group UK London}
}

@article{HannezoOpinion,
  title={Multiscale dynamics of branching morphogenesis},
  author={Hannezo, Edouard and Simons, Benjamin D},
  journal={Curr. Opin. Cell Biol.},
  volume={60},
  pages={99--105},
  year={2019},
  publisher={Elsevier}
}

@article{Fujishima,
  title={Principles of branch dynamics governing shape characteristics of cerebellar Purkinje cell dendrites},
  author={Fujishima, Kazuto and Horie, Ryota and Mochizuki, Atsushi and Kengaku, Mineko},
  journal={Development},
  volume={139},
  number={18},
  pages={3442--3455},
  year={2012},
  publisher={Company of Biologists}
}

@article{Yu,
  title={Simple rules determine distinct patterns of branching morphogenesis},
  author={Yu, Wei and Marshall, Wallace F and Metzger, Ross J and Brakeman, Paul R and Morsut, Leonardo and Lim, Wendell and Mostov, Keith E},
  journal={Cell Syst.},
  volume={9},
  number={3},
  pages={221--227},
  year={2019},
  publisher={Elsevier}
}

@article{Turing,
  title={The chemical basis of morphogenesis},
  author={Turing, Alan Mathison},
  journal={Bull. Math. Biol.},
  volume={52},
  number={1},
  pages={153--197},
  year={1990},
  publisher={Springer}
}

@article{Kreten,
  title={Traveling waves of an FKPP-type model for self-organized growth},
  author={Kreten, Florian},
  journal={J. Math. Biol.},
  volume={84},
  number={6},
  pages={42},
  year={2022},
  publisher={Springer}
}

@article{Sun,
  title={Turing mechanism underlying a branching model for lung morphogenesis},
  author={Xu, Hui and Sun, Mingzhu and Zhao, Xin},
  journal={PLoS One},
  volume={12},
  number={4},
  pages={e0174946},
  year={2017},
  publisher={Public Library of Science San Francisco, CA USA}
}

@book{BansayeMeleard2015,
  title     = {Stochastic Models for Structured Populations: Scaling Limits and Long Time Behavior},
  author    = {Vincent Bansaye and Sylvie M{\'e}l{\'e}ard},
  series    = {Mathematical Biosciences Institute Lecture Series},
  volume    = {1.4},
  publisher = {Springer},
  address   = {Cham, Switzerland},
  year      = {2015},
  isbn      = {978-3-319-21710-9},
  doi       = {10.1007/978-3-319-21711-6},
}

@article{Birkneral,
  title={Survival and complete convergence for a branching annihilating random walk},
  author={Birkner, Matthias and Callegaro, Alice and {\v{C}}ern{\`y}, Ji{\v{r}}{\'\i} and Gantert, Nina and Oswald, Pascal},
  journal={	Ann. Appl. Probab.},
  volume={34},
  number={6},
  pages={5737--5768},
  year={2024},
  publisher={Institute of Mathematical Statistics}
}

@article{Bramson,
  title={The survival of branching annihilating random walk},
  author={Bramson, Maury and Gray, Lawrence},
  journal={	Probab. Theory Relat. Fields},
  volume={68},
  number={4},
  pages={447--460},
  year={1985},
  publisher={Springer}
}

@article{CardyT,
  title={Theory of branching and annihilating random walks},
  author={Cardy, John and T{\"a}uber, Uwe C},
  journal={	Phys. Rev. Lett.},
  volume={77},
  number={23},
  pages={4780},
  year={1996},
  publisher={APS}
}

@article{Bertoin,
  title={Markovian growth-fragmentation processes},
  author={Bertoin, Jean},
  year={2017},
    journal={Bernouilli},
  volume={23},
  number={2},
  pages={1082--1101},
  year={1996},
  publisher={Bernoulli Society for Mathematical Statistics and Probabilit}
}

@article{BCK,
  title={Random planar maps and growth-fragmentations},
  author={Bertoin, Jean and Curien, Nicolas and Kortchemski, Igor},
  journal={Ann. Probab.},
  volume={46},
  number={1},
  pages={207--260},
  year={2018},
  publisher={JSTOR}
}

@article{SilvaPardo,
  title={Multitype self-similar growth-fragmentations},
  author={Da Silva, William and Pardo, Juan Carlos},
  journal={ALEA - Lat. Am. J. Probab. Math. Stat.},
  volume={21},
  pages={207--260},
  year={2024},
  publisher={IMPA}
}

@article{DJ,
  title={Probabilistic representations of fragmentation equations},
  author={Deaconu, Madalina and  Lejay, Antoine },
  journal={Probab. Surveys},
  volume={20},
  pages={226-290},
  year={2023},
  publisher={ Institute of Mathematical Statistics, Bernoulli Society for Mathematical Statistics and Probability}
}

@article{growth,
  title={A spatially-dependent fragmentation process},
  author={Callegaro, Alice and Roberts, Matthew I.},
  journal={	Probab. Theory Relat. Fields},
  volume={192},
  pages={163-266},
  year={2024},
  publisher={Springer}
}

@PhdThesis{Dadoun,
  author = {Benjamin Dadoun},
  title  = {{Some aspects of growth-fragmentation}},
  school = {University of Zurich },
  year   = {2020},
  url   = {https://www.zora.uzh.ch/handle/20.500.14742/178336},
  language  = {English}
}

@article{Dikec,
  title={Hyphal network whole field imaging allows for accurate estimation of anastomosis rates and branching dynamics of the filamentous fungus Podospora anserina},
  author={Dikec, J and Olivier, A and Bob{\'e}e, C and D’angelo, Y and Catellier, R and David, P and Filaine, F and Herbert, S and Lalanne, Ch and Lalucque, Herve and others},
  journal={Sci. Rep.},
  volume={10},
  number={1},
  pages={3131},
  year={2020},
  publisher={Nature Publishing Group UK London}
}

@article{Meinhardt,
  title={Morphogenesis of lines and nets.},
  author={Meinhardt, Hans},
  journal={Differentiation; research in biological diversity},
  volume={6},
  number={2},
  pages={117--123},
  year={1976}
}
\end{document}